\documentclass[11pt]{article}
\usepackage{amssymb,amsmath,amsthm,graphicx,latexsym}

\newtheorem{theorem}{Theorem}[section]
\newtheorem{conjecture}{Conjecture}[section]

\newtheorem{remark}{Remark}

\begin{document}
\title{\textbf {\Large On the Dynamics of Complex Quadratic Families under Holomorphic and Non-Holomorphic Singular Perturbations}}
\author{
    \textbf{Haitao Shang}\footnotemark[1] \footnotemark[2]}

\renewcommand{\thefootnote}{\fnsymbol{footnote}}
\footnotetext[1]{\emph {Department of Mathematics, University of Minnesota, MN, USA}}
\footnotetext[2]{Email: htshang.research@gmail.com}

\date{}

\maketitle
\abstract{This work surveys the topological and statistical properties of {\it real} quadratic maps and investigates the {\it complex} quadratic maps under holomorphic and non-holomorphic singular perturbations.} 

\tableofcontents

\clearpage
\section{Introduction}

It is well known that differential equations can be used to model many phenomena in nature. Following Newton's fundamental discovery, which successfully explained the motion of a two-body system, this perspective was widely accepted by most scientists. Over the next two centuries, Newton's successors continued his methodology and attempted to construct a unified linear model to describe the entire natural world. However, all attempts to explicitly and analytically solve the three-body problem ended in failure. Some physical scientists gradually recognized the limitation of linearization: the local linearization of nonlinear problems causes the ``full truth" of the system to be overshadowed by the ``partial beauty" of local linear properties.

At the end of the 19th century, H. Poincar\'{e} proposed an entirely new viewpoint: rather than analytically solving the three-body problem, one should qualitatively describe the behavior of solutions to the differential equations. Generally, qualitative behavior includes fixed points (stationary solutions), periodic orbits (periodic oscillations), orbits exhibiting forward or backward asymptotic behavior, and orbits exhibiting both forward and backward asymptotic behavior (homoclinic orbits) \cite{Lyubich I, Strogatz}. Poincar\'{e} initially believed that these homoclinic orbits occupied most of the phase space; however, this incorrect conjecture eventually inspired other forms of qualitative analysis, such as recurrent trajectories, instability, mixing, and entropy, which collectively influenced the development of chaos theory.
This qualitative analysis perspective later evolved into the core paradigm of modern complex systems theory. Complex systems emphasize that the behavior of the whole cannot be deduced solely from the properties of individual components or local interactions; rather, it emerges from the interplay of multiscale interactions and nonlinear couplings within the system. Such systems often exhibit emergent phenomena, where local rules fail to directly account for global patterns. For example, in ecological, economic, or climate systems, the simple behavior of individual components can give rise to highly unpredictable global dynamics. Complex systems theory not only highlights nonlinearity and high dimensionality, but also focuses on adaptability, emergent structures, and self-organization—features that provide an essential conceptual foundation for subsequent developments in chaos theory.

Moreover, the dynamics of complex systems are often constrained by network structures and feedback mechanisms. Positive feedback can amplify local perturbations, leading to cascading effects or sudden transitions, whereas negative feedback tends to stabilize the system or generate oscillatory behavior. Many complex systems also display scale invariance or approximate power-law distributions, implying that the probability of extreme events is far higher than would be expected from exponential decay, and that there exists a nontrivial link between local events and global patterns. By incorporating these concepts, complex systems theory can explain how nonlinear microscopic rules generate diverse dynamical patterns across scales, including chaos, fractals, self-organized criticality, and synchronized or collective behaviors in complex networks. This cross-scale, interdisciplinary framework provides a deeper understanding of modern dynamical systems and chaos theory, and guides the analysis of real-world phenomena that are inherently difficult to predict.

Poincar\'{e}'s successor, G. Birkhoff, initiated research in topological dynamics and ergodic theory \cite{Strogatz}. Concurrently, J. Julia and P. Fatou began studying complex dynamics and laid the foundations of this field \cite{Alexander}. In the following thirty years, A. Kolmogorov, V. Arnold, and J. Moser conducted in-depth studies on the stability of Hamiltonian systems, building upon Birkhoff's unfinished work and opening a new era in dynamical systems \cite{Hao}. Progress in complex dynamical systems during this period was relatively slow due to the lack of powerful computational tools \cite{Alexander}.
In 1963, E. Lorenz presented a classical example in his seminal article, ``Deterministic Non-periodic Flow" \cite{Lorenz}. He proposed a set of three-dimensional differential equations to describe atmospheric convection and discovered that the simulation results were extremely sensitive to initial conditions, a phenomenon later termed the ``Butterfly Effect." In 1971, D. Ruelle and F. Takens proposed a new theory of turbulence in dissipative systems based on preliminary considerations of strange attractors \cite{Ruelle}. A few years later, T. Li and J. Yorke published ``Period Three Implies Chaos" \cite{Yorke}, which marked the first formal introduction of the term ``chaos" (although A. Sarkovskii had independently proposed a similar, more general theorem \cite{Devaney III}). Simultaneously, R. May demonstrated chaotic phenomena in population dynamics using the logistic equation \cite{May}. Other researchers, such as A. Winfree \cite{Winfree} and M. Feigenbaum \cite{Feigenbaum I, Feigenbaum II}, also significantly advanced the field during this period. The 1970s thus represented a pivotal decade in the development of chaos theory. In the 1980s, with rapid advances in computer science, many new theories in dynamical systems, particularly complex dynamics, were discovered. B. Mandelbrot established a new branch, ``fractal geometry," based on numerical analysis and graphical technology \cite{Mandelbrot}.
Because chaotic motions are non-periodic, the effective period of a chaotic system can be considered infinite. Consequently, chaotic systems are intrinsically linked to the concept of ``infinity." A central problem is determining whether a finite physical process is truly chaotic or simply exhibits very long periodicity. Traditionally, a physical system is considered chaotic if it satisfies the mathematical definitions of chaos, namely, exhibiting positive Lyapunov exponents and sensitivity to initial conditions.

Further studies have revealed that many complex and chaotic systems exhibit power-law distributions, indicating that the probability of extreme events decays much more slowly than in exponential distributions. This means that rare but large-scale events (fat tails) can play a disproportionate role in system behavior. 
Power-law behaviors imply scale-invariance, where statistical patterns observed at small scales can provide insight into large-scale dynamics. 
Such scaling properties are closely related to fractal geometry and multi-scale correlations observed in chaotic attractors, making them central to understanding emergent phenomena across physical and biological processes \cite{Shang2023generic, Shang2024Probing}, geological history \cite{Shang2024Percapita, Shang2024powerisotopes, shang2024scaleinvariant}, and life evolution \cite{Shang2024Scaling, Shang2025Accelerations}, and other dynamical systems.
Power-law distributions often arise in conjunction with multi-level feedback mechanisms. Positive feedback can amplify small perturbations, triggering cascading effects, tipping points, or abrupt regime shifts, whereas negative feedback constrains fluctuations, maintaining stability or producing oscillatory dynamics. 
Positive and negative feedback mechanisms widely exist in natural systems and processes \cite{Shang2023Oxidative, Shang2023Mineral, Shang2023Dichotomous}. The interplay of these feedbacks can lead to self-organized criticality, a state in which systems naturally evolve toward critical points where minor perturbations produce avalanches of all sizes. Understanding how feedback interacts with nonlinear network dynamics is therefore crucial for linking classical chaos theory with modern complex systems, providing insight into predictability, resilience, and emergent behavior across natural and artificial systems.

Nowadays, dynamical systems theory encompasses a wide range of perspectives, and chaos is recognized as one of the most significant scientific discoveries of the twentieth century. The generally accepted objective involves ``the study of asymptotic behaviour of almost all orbits in observation of representative finite parameter families" \cite{Lyubich I}, where ``representative" refers to parameter values producing nontrivial maps. R. Devaney \cite{Devaney III} provided a classical definition of chaos: let $X$ be a metric space and $F: X \rightarrow X$ a continuous map. Then $F$ is chaotic if (1) $F$ is transitive; (2) $F$ exhibits sensitive dependence on initial conditions; (3) the periodic points are dense in $X$.
Among these three conditions, sensitivity is central. In fact, transitivity and the density of periodic points imply sensitivity \cite{Robinson, Banks}. Other qualitative or quantitative measures, such as topological entropy and power spectrum, provide additional characterizations of chaos \cite{Robinson}. In one-dimensional systems, transitivity alone implies the other two conditions.

The dynamics of some quadratic families have been extensively studied, revealing many elegant results, whereas the dynamics of other families remain largely unexplored. In this paper, we focus on real and complex quadratic maps. The paper is organized as follows: Section 2 reviews representative one- and two-dimensional quadratic maps from both topological and statistical perspectives. Section 3 summarizes prior results on complex quadratic maps under holomorphic and nonholomorphic singular perturbations and investigates a special case emphasizing the real line, with original contributions presented in Sections 3.2 and 3.6. Section 4 provides a summary.

\clearpage
\section{Topological and Statistical Properties of Dynamics of Real Quadratic Maps}
\subsection{Preliminary}
\item In this section, we consider the dynamics of one- and two-dimensional real quadratic maps with two examples: $F_{a} (x)= 1-ax^{2}$ and the H\'{e}non map. The early study of the dynamics of real quadratic maps is motivated by R. May's research on the population dynamics in 1976 \cite {May}. Much progress has been made towards the analysis of this type of map, which now is treated as a representative model of chaotic dynamics. However, the behaviour of a chaotic system is sensitive to the initial conditions, it is hard and even impossible to study all individual orbits; to solve this technical difficulty, Ergodic theory is suggested to study the long-term behaviours of the typical orbits.

\item \textbf{Ergodic theory} is concerned with the distributional properties of the typical orbits of a dynamical system throughout the phase space, and these statistical properties of orbital distributions are described in terms of measure theory, especially the invariant measure under a transformation or flow. With an invariant measure, many elegant theorems about the dynamical behaviours have been claimed. Two elegant and well-known example are the Poincar\'{e}'s Recurrent Theorem and the Birkhoff Ergodic Theorem.

\begin{remark}
Roughly speaking, the Ergodic theorem studies how large-scale phenomena nonetheless create non-random regularity. The term ``ergodic" originates from Greek words: ``ergon (work)" and ``odos(path)" \cite{Walters}. This term was created by L. Boltzmann in statistical mechanics and it included a hypothesis: ``for a large system of interacting particles in equilibrium, the time average along a single trajectory equals the space average." Unfortunately, this hypothesis was false; but the property a system needs to satisfy to ensure these two quantities (time means and phase means of real-valued functions) to be equal is called ``ergodicity" nowadays. And a modern version of ergodic theory is: the study of long-term average behavior of systems that are evolving with time in their phase spaces \cite{Walters}.      
\end{remark}

\item Before further discussion, let us define that $\mathcal{B}$ is the Borel $\sigma$-algebra of $X$, and $\mu$ is an invariant measure under a transformation $T: X \rightarrow X$, in which $X$ is equipped with some structures (for examples, $X$ is a topological space or a smooth manifold) and $T$ preserves these structures (for example, $T$ is a homeomorphism, diffeomorphism, or a continuous transformation). And $h: X \rightarrow \mathbb{R}$ is defined as a Dirac measure, which is also named as an observable or a characteristic function. In addition, one necessary term should be introduced is ``absolute continuity". One general definition this term, which has several equivalent definitions in different cases \cite{Nielsen}, is: a measure $\mu$ is called absolutely continuous with respect to a measure $\nu$ if  $\forall$ measurable set $E$, $\mu(E) = 0 \implies \nu(E) = 0$.\\
(1) \textbf {Poincar\'{e}'s Recurrent Theorem} states that 

\begin{theorem}
Given a measurable set $A \in \mathcal{B}$ with $\mu (A) > 0$ in a probability space $(X, \mathcal{B}, \mu)$, we have
\begin {equation*}
\mu (\{ a \in A: \exists N \in \mathbb{N}, \forall n >N, T^{n} (a) \notin A \}) = 0.
\end{equation*}
\qed
\end{theorem}
In other words, almost every point in A returns infinitely often back into A under forward iteration by $T$. \\
(2) \textbf {Birkhoff Ergodic Theorem} claims that
\begin{theorem}
for almost every $x \in X$,
\begin{equation*}
\lim_{n \to +\infty} \frac{1}{n} \sum_{i=1}^{n} h (T^{i}(x)) = \int h d\mu ,
\end{equation*}
in which $\mu$ is an invariant measure in a probability space $(X, \mu)$ (note not $(X, \mathcal{B}, \mu )$). 
\qed
\end{theorem}
Intuitively, we can state that ``time-averages equals space-averages almost everywhere" for an ergodic endomorphism.

\begin{remark}
We give the definitions of invariant measure and ergodic transformation here:\\
(1) Invariant measure: An invariant measure on a measurable space $(X, \mathcal{B}, \mu)$ with respect to a measurable transformation $T$ of this space is a measure $\mu$ on $\mathcal{B}$ for which $\mu (A) = \mu (T^{-1}(A)) $ for all $A \in \mathcal{B}$. (2)Ergodic transformation: Let $T: X \to X$ be a measure-preserving transformation on a measure space $(X, \mathcal{B}, \mu)$, with $\mu (X) = 1$.Then $T$ is an ergodic transformation if for every $A \in \mathcal{B}$ with $T^{-1}(A) = A$ either $\mu(A) = 0$ or $\mu(A) = 1$.

\end{remark}

\begin {remark}
It is worth to mention that if $\mu$ is an invariant Borel measure, that is, $\mu $ is an invariant measure in a probability space $(X, \mathcal{B}, \mu )$, then we claim that ``time-averages equal space-averages everywhere" for an ergodic endomorphism.
\end{remark}

\begin {remark}
One should note the statement of the Birkhoff Ergodic Theorem is only with respect to the invariant measure $\mu$ \cite{Eckmann}; therefore, a set with a full $\mu $-measure may have a 0 Lebesgue measure. In this sense, the invariant measure $\mu$ may lack of physical meaning.
\end{remark}

\item To overcome the technical problem mentioned in the Remark 5, Y. Sinai, D. Ruelle, and R. Bowen introduced the \textbf{physical measure}, which is an ergodic probability measure that is absolutely continuous with respect to Lebegue measure. The other important measure in dynamical systems is \textbf{SRB measure}, which is named after Y. Sinai, D. Ruelle, and R. Bowen but first formally suggested by P. Collet and J. Eckmann. Physical measure is an ergodic SRB measure with no zero Lyapunov exponents \cite{young}. See reference and \cite{Eckmann} \cite{young} for more discussion about the properties of SRB measure and physical measure, and the relations between them.

\item In addition to the statistical approach,we have the other traditional way to understand the dynamics of a system, namely, the topological viewpoint. From the approach, we study the topological properties, such as the hyperbolicity and topological entropy, of dynamical systems.

\item Let $X$ be a compact metric space with metric $d$ and $F: X \rightarrow X$ be a continuous transformation. For $\epsilon > 0$ and $n \in \mathbb{Z}^+$, we say $E \subset X$ is an $(n, \epsilon)$-separated set if for every $x,y \in E$ there exists $0 \leq i < n$ such that $d(f^i(x), f^i(y)) > \epsilon$. 
Then, the \textbf{topological entropy} of $f$, which we denote $h_{\text{top}}(f)$, is defined as
\begin{equation*} 
h_{\text{top}}(f) = \lim_{\epsilon \to 0}\lim_{n \to \infty} \sup \dfrac{1}{n}\log N(n, \epsilon),
\end{equation*}
where $N(n, \epsilon)$ represents the maximum cardinality of all $(n, \epsilon)$-separated sets.

\begin{remark}
Positive topological entropy implies topological chaos in the Li-Yorke sense, in which there exists an uncountable scrambled set  \cite{Lyubich I}. A set $S \subseteq X$ is scrambled if every pair $(x, y)$ of distinct points in $S$ satisfies $ \lim_{n \to \infty} \inf d(f^n(x), f^n(y)) = 0$ and $\lim_{n \to \infty} \sup d(f^n(x), f^n(y)) > 0$ .
\end{remark}

Meanwhile, the other concept, \textbf{hyperbolicity}, is defined as follows: a compact set $X \subset M$, where $M$ is a compact manifold, is hyperbolic if\\
(1)$X$ is invariant under a diffeomorphism $F$, that is, $F(X) \subset X$; \\
(2)for all $x \in X$, the tangent space $T_{x}M$ has a continuous splitting, $T_{x}M = E^{s} \oplus E^{u}$, where stable manifold $E^{s}$ is uniformly contracting and unstable manifold $E^{u}$ is uniformly expanding under the derivatives. \\
In different cases, the hyperbolicity is called uniform, semiuniform and nonuniform. Based on the hyperbolicity, we can obtain more topological concepts of dynamical systems, such as shadowing, homoclinicity, Markov partition and so forth. See \cite{Robinson} and \cite{Eckmann} for more discussion.

\item From the perspective of statistics and topology, we will review and summarize some theories about the one- and two-dimensional real quadratic maps in the rest of this section.  

\clearpage
\subsection{Dynamics of a Real Quadratic Family $F_{a}: \mathbb{R} \rightarrow \mathbb{R},  F_{a} (x)= 1-ax^{2}$}
\item Firstly we consider a map in the one-dimensional case, that is, the quadratic family
\begin{equation*}
F_{a}: \mathbb{R} \rightarrow \mathbb{R},  F_{a} (x)= 1-ax^{2}
\end{equation*}
where $a \in \mathbb{R}$.
 Let us restrict our discussion about of $F_{a}$ when $a \in ( 0, 2] $ and $x \in [ -1,  1] $, since the dynamics is simple and well-understood outside the parameter interval or the domain. In this case, $F_{a}: [-1, 1] \rightarrow [-1, 1]$ is an $\textbf{S-unimodal}$ map, since it is of class $C^{3}$ and has a negative Schwarzian derivative \cite {Devaney VI}:

\begin{equation*}
S(F_{a}) = \dfrac {{F_{a}}'''}{{F_{a}}'} - \dfrac {3}{2} {\left( \dfrac {{F_{a}}''}{{F_{a}}'} \right) }^{2} < 0.
\end{equation*}

\item  Now we introduce a theorem about the dynamics of $F_{a}$ with the statistical viewpoint.
\begin{theorem}
\emph{(Jakobson Theorem \cite{Jakobson}, 1981)}
There is a positive Lebesgue measure set of parameters $a \in (0, 2]$ for which $F_{a}$ has an absolutely continuous ergodic measure $\mu _{a}$. \qed
\end{theorem}
\item According to the discussion above, we know that this measure $\mu _{a}$ is a physical measure. Now we can apply the property of physical measure to study the dynamics of $F_{a}$ when $a \in (0, 2]$. Before further discussion, let us firstly introduce an elegant theorem:
\begin{theorem}
\emph{(Oseledec Theorem \cite{Eckmann}, 1968)}
Let $\mu$ be an ergodic invariant measure for a diffeomorphism $F$ of a compact manifold $M$. Then for $\mu$-almost every initial condition $x$, the sequence of symmetric nonnegative matrices
\begin{equation*}
\sqrt[2n]{(D_{x} F^{n})^{T} (D_{x} F^{n})},
\end {equation*}
where $D_{x} F^{n}$ denotes the differential of the map $F^{n}$ at the point x, converges to a symmetric nonnegative matrix $\Lambda $ (independent of x). Denote by $\lambda _{0} > \lambda _{1} > . . . > \lambda _{k}$ the strictly decreasing sequence of the logarithms of the eigenvalues of the matrix  (some of them may have nontrivial multiplicity). These numbers are called the Lyapunov exponents of the map f for the ergodic invariant measure $\mu$. For $\mu$-almost every point x there is a decreasing sequence of subspaces
\begin {equation*}
M = E_{0}(x) \supset E_{1}(x) \supset \dotsb \supset E_{k}(x) \supset E_{k+1}(x) = \{\mathbf{0}\},
\end {equation*}
satisfying ($\mu $-almost surely) $D_{x}F E_{i}(x) = E_{i}(F(x))$ and for any $i \in \{0,  \dotsc , k\} $ and any initial error vector $\mathbf{h} \in E_{i}(x) \setminus E_{i+1}(x)$ one has
\begin {equation*}
\lambda_{i} = \lim_{n\rightarrow \infty}\dfrac{1}{n}\log||D_{x}F^{n} \mathbf{h}||.
\qed
\end {equation*}
\end{theorem}

\item Let's now use ''$F'$" to denote the first derivative of $F$, then it is easy to show that
\begin{equation*}
\dfrac{1}{n}\log||{F^{n}} '(x)|| = \dfrac{1}{n} \sum_{i=1}^{n} (\log||F'(F^{i}(x))||),
\end{equation*}
the right hand side is a temporal average \cite{Eckmann}. Meanwhile, since $\log|{F_{a}}'|$ is $\mu_{a}$-integrable and $\int {\log|{F_{a}}'|} d\mu_{a} > 0$, then by the Birkhoff Ergodic Theorem,
\begin{equation*}
\lim_{n \to +\infty}\dfrac{1}{n}\log|{F_{a}^{n}} '(x)| = \lim_{n \to +\infty} \dfrac{1}{n} \sum_{i=1}^{n} (\log|{F_{a}}'(F_{a}^{i}(x))|) = \int {\log|{F_{a}}'|} d\mu _{a}.
\end{equation*}

Thus, the Lyapunov exponent mentioned in the Oseledec Theorem in the one-dimensional case is
\begin{equation*}
\lambda _{a} = \int {\log|{F_{a}}'|} d\mu _{a} > 0,
\end{equation*}
which implies the sensitive dependence on the initial conditions, a primary feature of chaotic behaviour.

\item As mentioned before, there are two traditional approaches to study the dynamics of a system, one is the statistical or ergodic, while the other is called topological or differential-geometric \cite{Eckmann}. Now let us introduce a theorem about the dynamics of $F_{a}$ from the topological approach:

\begin{theorem}
\emph{(J. Graczyk and G. Swiatek \cite{Graczyk}, 1997)}
There is an open dense set $S \subset (0,2]$, for Lebesgue-almost every point $a \in S$, $F_{a}$ has a periodic attracting orbit. \qed
\end{theorem}

\item We have shown two properties of the map $F_{a}$ (Theorem 2.1 and Theorem 2.3) so far. Indeed there exist other behaviors when $a \in (0,2]$ and $x \in [-1, 1]$, such as the disappearance of attracting period orbits, the vanishment of physical measures, and so forth. Based on the two properties stated above, M. Lyubich described a beautiful global picture:
\begin{theorem}
\emph{(M. Lyubich \cite{Lyubich III}, 2002)}
For Lebesgue-almost every $a \in (0, 2]$,  the map $F_{a}$ has either a periodic attracting orbit or an absolutely continuous ergodic measure. \qed
\end{theorem}

\item A unimodal map $F: X \rightarrow X$ is called \textbf{stochastic} if it has an invariant measure which is absolutely continuous with respect to the Lebesgue measure on $X$. Thus $F_{a}$ is a stochastic map for $a \in (0, 2]$. The existence of the absolutely continuous invariant measure is related to the rate of expansion along the orbit of the critical point $c_{0}$:

\begin{theorem}
\emph{(T. Nowicki \cite{Nowicki I}, 1988)}
A map $F: X \rightarrow X$ is stochastic if its expansion rate is exponential:
\begin{equation*}
DF^{n}(c_{0}) \geq C e^{\lambda n},
\end{equation*}
where $n \in \mathbb{N}$, the constant $C > 0$, and the Lyapunov exponent $\lambda > 0$. \qed
\end{theorem}

For an S-Unimodal function, such as $F_{a}$, by replacing the exponential rate with the summability condition, one can prove \cite{Nowicki II} that $F_{a}$ is stochastic since
\begin{equation*}
\sum_{i=0}^{\infty } {\| D{F_{a}}^{i}(c_{0})\| }^ {-1/ 2} < \infty .
\end{equation*}
Furthermore, one can show \cite{Nowicki I} that
$\exists$ a constant $K < \infty$ such that $\forall \epsilon >0$
\begin{equation*}
| {F_{a}} ^{n} (c-\epsilon , c+\epsilon) | < K\epsilon.
\end{equation*}

\begin{remark}
There are other terminologies that are equivalent to ``stochastic" and ``regular": \textbf{uniformly hyperbolic} and \textbf{non-uniformly hyperbolic}. In the sense of Pesin theory (that is, the invariant measure automatically has a positive characteristic exponent \cite {Lyubich III}), stochastic quadratic maps can be called non-uniformly hyperbolic. Meanwhile, since regular quadratic maps are uniformly expanding outside the basin of the attracting cycle, they are also called uniformly hyperbolic. Therefore, we can claim that almost any real quadratic map is hyperbolic. See more discussion in references \cite {Lyubich I} and \cite {Lyubich II}.
\end{remark}

\item From the discussions above, we obtain the other version of the Theorem 2.5:
\begin{theorem}
\emph{(M. Lyubich \cite{Lyubich I}\cite{Lyubich III}, 2000)}
For Lebesgue-almost every $a \in (0, 2]$,  the map $F_{a}$ is either stochastic (non-uniformly) or regular (uniformly). \qed
\end{theorem}

\item In order to magnify and study a class of dynamical systems with small-scale structures, we can apply the method of renormalization to act on this class of dynamical systems. Mathematically, we usually construct the renormalization operator as a return map and there are several methods of construction according to the class of dynamical systems we are focusing on. For a stochastic map, the absolutely continuous invariant measure is supported on a cycle of intervals with disjoint interiors. If a unimodal map has such a cycle of intervals, we call this map is \emph{renormalizable}. How many times a map can be renormalized depends on the map itself. According to the number of times we can renormalize the quadratic maps, we classify them as ``at most finitely" or ``infinitely" renormalizable. For more details and discussion, see \cite{Lyubich I} and \cite {Lyubich II}.

\clearpage
\subsection{Dynamics of the H\'{e}non Map}
\item Now let us consider the two-dimensional case and introduce an unfolding of the quadratic family, the H\'{e}non Map, $F_{a,b}: \mathbb{R}^{2} \rightarrow \mathbb{R}^{2}$, where $F_{a,b} (x, y)= (1-ax^{2}+y, bx)$ ($a, b \in \mathbb{R}$), which was suggested as a simplified model of the Poincar\'{e} map of the Lorentz system and is chaotic when the two parameters take the canonical values a=1.4 and b=0.3 (where a strange attractor, H\'{e}non attractor, emerges)(the upper left figure in Fig.1) \cite{H'{e}non}.

\begin{figure}[h]
\centering
\centerline{
\includegraphics[scale=0.4]{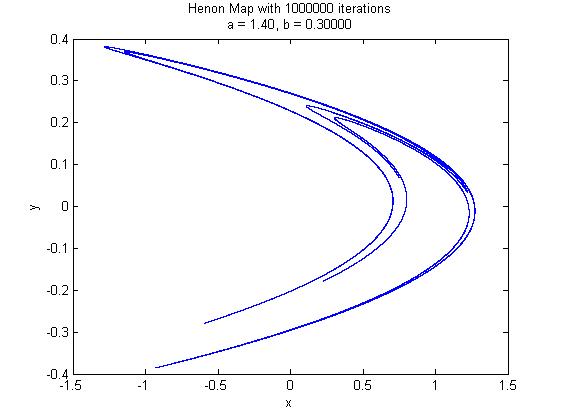}
\includegraphics[scale=0.4]{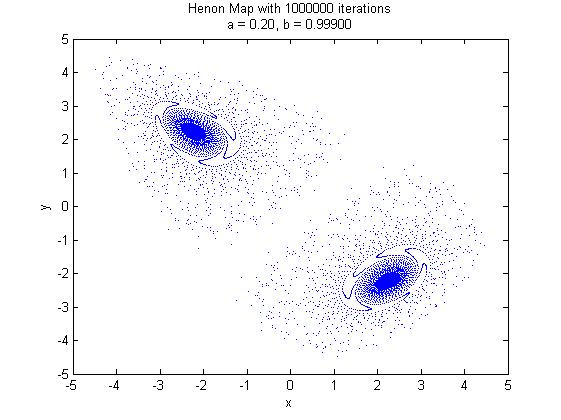}}
\centerline{
\includegraphics[scale=0.4]{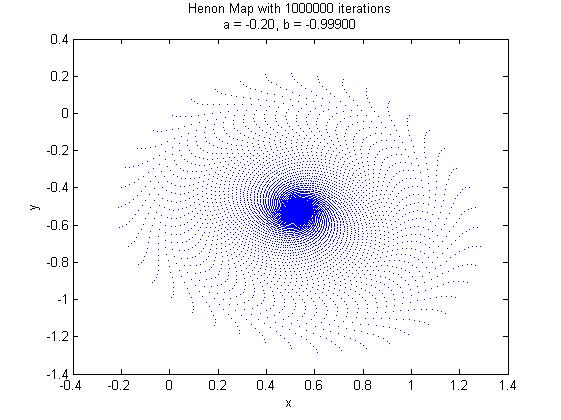}
\includegraphics[scale=0.4]{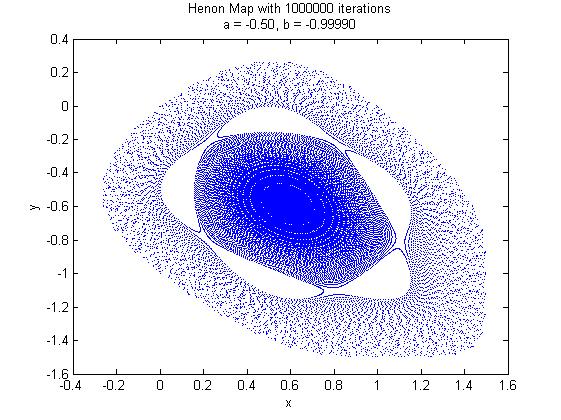}}
\caption{\small An orbit for a of H\'{e}non Map for Different Parameters a and b}
\end{figure}

\item One can easily show that the H\'{e}non map is injective, and its inverse is also injective when $b \neq 0$. Meanwhile, if we take different values of the parameter a in terms of the parameter b, then we have the following theorem:
\begin{theorem}
For the H\'{e}non map $F_{a,b}: \mathbb{R}^{2} \rightarrow \mathbb{R}^{2}$, where $F_{a,b} (x, y)= (1-ax^{2}+y, bx)$ $(a, b \in \mathbb{R})$, \\
(1) when $a< -\dfrac{1}{4}{(1-b)}^2$, there is neither fixed nor periodic point;\\
(2) when $-\dfrac{1}{4}{(1-b)}^2 < a< \dfrac{3}{4}{(1-b)}^2$, there are two fixed points, one is attracting, while the other is repelling;\\
(3) when $\dfrac{3}{4}{(1-b)}^2 < a$, there are two attracting periodic points. \qed
\end{theorem}

\begin{figure}[h]
\centering
\centerline{
\includegraphics[scale=0.35]{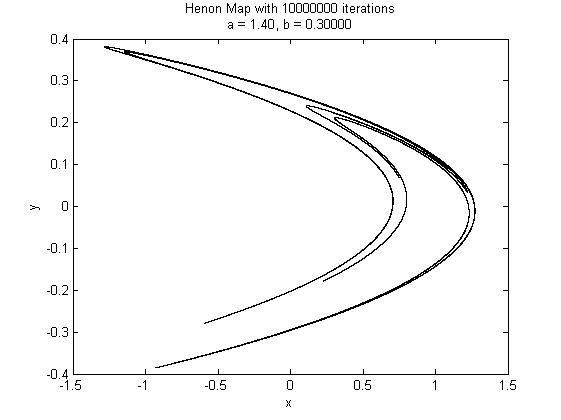}
\includegraphics[scale=0.35]{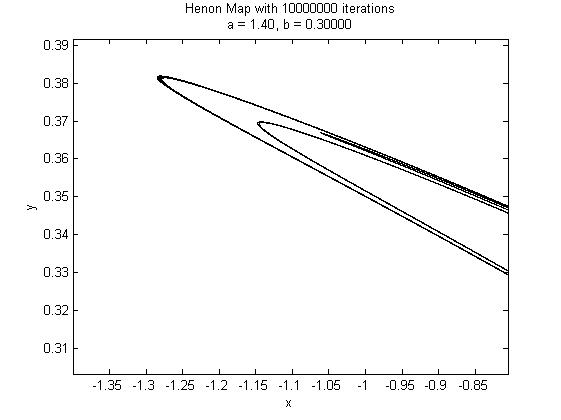}
\includegraphics[scale=0.35]{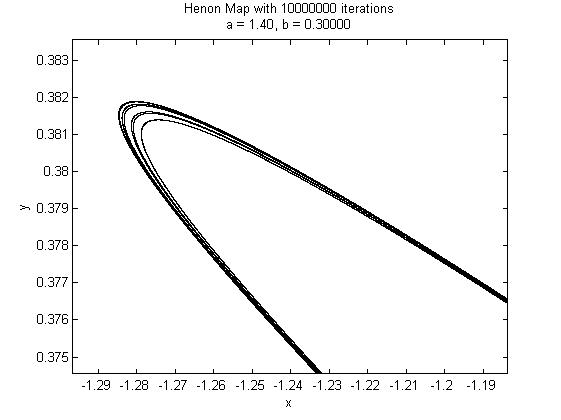}}
\centerline{
\includegraphics[scale=0.35]{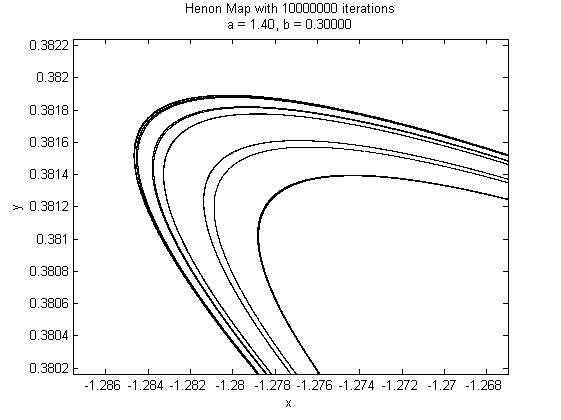}
\includegraphics[scale=0.35]{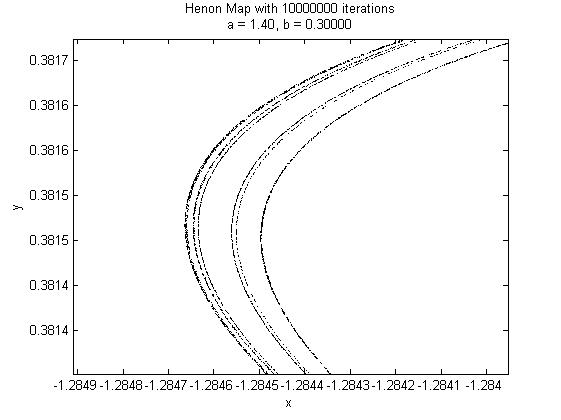}
\includegraphics[scale=0.35]{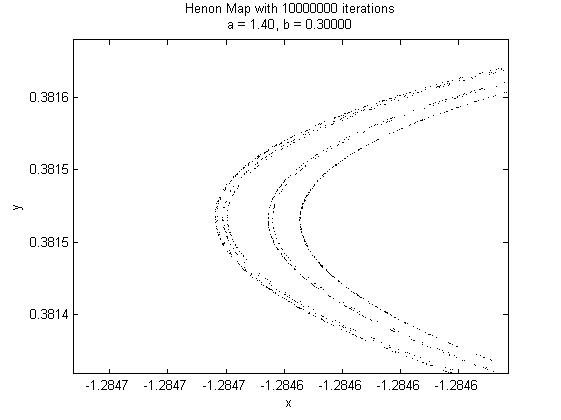}}
\caption{\small Self-Similarity of the H\'{e}non Attractor ($a = 1.4$ and $b = 0.3$)}
\end{figure}

\begin{figure}[h]
\centering
\centerline{
\includegraphics[scale=0.4]{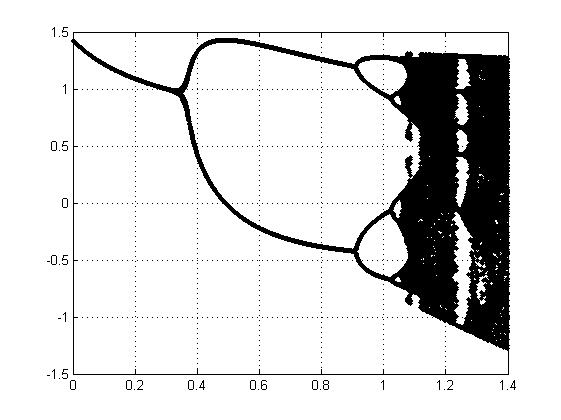}
\includegraphics[scale=0.4]{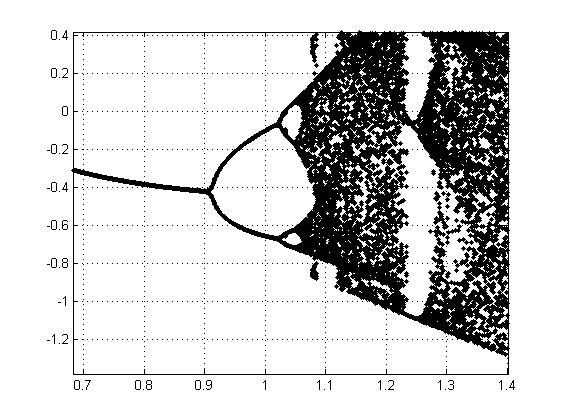}}
\centerline{
\includegraphics[scale=0.4]{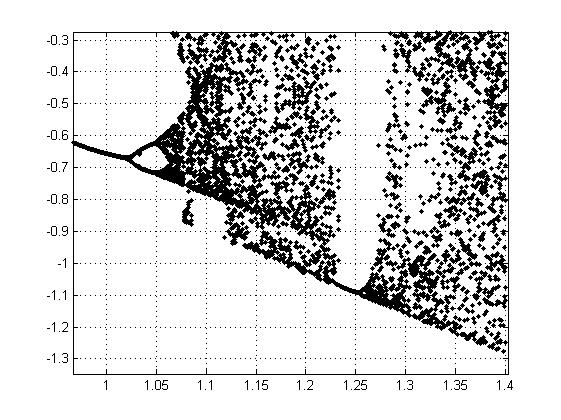}
\includegraphics[scale=0.4]{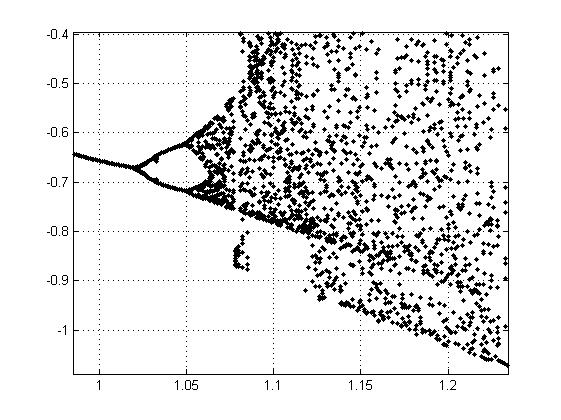}}
\caption{\small The Bifurcation Diagrams of H\'{e}non Map when $b = 0.3$, $x_{0} = 0$ and $y_{0} = 0$.}
\end{figure}

\item Let us consider the convergent values of $x$ corresponding to different values of $a$. When $b = 0.3$, if setting the initial values $x_{0} = 0$ and $y_{0} = 0$, one can show that the value of $x$ in the H\'{e}non map receives real values when $a \in [0, 1.4]$, and the corresponding converged value of $x$ is $x \in [-1.5, 1.5]$ (the graphs in shown in the Fig.3). Based on the numerical experiment, we can see that when $a \in [0, 0.32]$, the sequence of he points on the orbit of $x_{0}$ converges to a fixed point independent on the initial values $x_{0}$ and $y_{0}$. And when the value of $a \in [0.32, 0.9]$, this sequence converges to a periodic orbit of period two. If we change the value of parameter b to be $b = 0.4$, then the we will see the points of period one, two and four when $a$ is 0.2, 0.5 and 0.9, respectively. From the analysis above, we find that not only the H\'{e}non map has fractal structures, but also different chaotic attractors can exist simultaneously for a range of values of parameter $a$.  

\item Now we consider the topological and statistical properties of the H\'{e}non map. M. Benedicks, L. Carleson and L. Young \cite{Bendicks} \cite{young I} have provided a global of the dynamics of the H\'{e}non map from these two perspectives:
\begin{theorem}
\emph{(M. Benedicks, and L. Carleson \cite{Bendicks}, 1991; M. Benedicks, and L. Young \cite{young I}, 1993)}
There exists a positive Lebesgue measure set S of parameters such that for each $(a,b) \in S$ the H\'{e}non map $F_{a,b}$ possesses the following properties: \\
(1) there exists an open set $U \subset \mathbb{R}^{2}$ such that $\overline{F_{a,b}(U)}\subset U$ and $\Lambda = {\bigcap}_{n=0}^{\infty} {F_{a,b}}^{n}(U)$ attracts all orbits of $x \in U$; \\
(2) there is $x_{0} \in \Lambda$ whose orbit is dense in $\Lambda$, and there exists $c > 0$ such that $||D {F_{a,b}}^{n}(x_{0})|| \geq e^{cn}$ for all $n \geq 1$; \\
(3) $F_{a,b}$ has a unique physical measure on $\Lambda$. 
\qed
\end{theorem}

\begin{remark}
The second item in this theorem implies there exists a positive Lyapunov exponent in a dense orbit under the H\'{e}non map, which means that the attractor is sensitively dependent on the initial conditions for the parameters in the set $S$.  
\end {remark}

\item According to the discussion above, we already know that there exists a strange attractor for the H\'{e}non map. To measure the extent of the chaoticity of a system, there are several differential approaches, such as topological entropy and mixing. Here we restrict our discussion to the mixing of the H\'{e}non map; see \cite{Alessandro} for more discussion about the topological entropy of the H\'{e}non map.  

\item Roughly speaking,``mixing" means ``asymptotically independent". Firstly, we introduce a concept, \textbf{cross correlation function}:
\begin{equation*}
C_{g,h}(n) = \int g(x)h(f^{n}(x))d\mu (x) - \int g(x) d\mu (x) \int h(f^{n}(x)) d\mu (x),
\end{equation*}
where $g$ and $h$ are two square integrable observables. If we choose these two observables as the characteristic functions, namely, $g = \chi _{A} $ and $h = \chi _{B}$, then the cross correlation function takes the following form:
\begin{equation*}
C_{\chi _{A}, \chi _{B}}(n) = \int _{A} \chi _{B} \circ f^{n}(x)d\mu - \mu (A) \mu (B) = \mu (A \cap f^{n}(B)) - \mu (A)\mu (B). 
\end{equation*}
If the system loses the memory of the initial conditions after a long period of time, then we can expect that $C_{\chi _{A}, \chi _{B}}(n)$ approaches 0 and obtain the definition of ``mixing". Mathematically, for a dynamical system, the $f$-invariant measure $\mu$ is \textbf{mixing} if for any measurable subsets $A$ and $B$ in the phase space, we have
\begin{equation*}
\lim_{n \to \infty} \mu (A \cap f^{n}(B)) = \mu (A)\mu (B). 
\end{equation*}
It is worth to mention that mixing implies ergodicity, but the converse is not true \cite{Eckmann}. Based on the concepts introduced above, the \textbf{mixing rate} of a system is related to the decay of the $C_{g,h}(n)$. We say the decay of the correlation is exponential if $h(f^{n})$ and $h$ become uncorrelated exponentially fast as $n$ tends to infinity. 

\item Now we introduce a theorem about the extent of chaoticity in terms of the mixing rate of H\'{e}non map \cite{young I} :
\begin{theorem}
\emph{(M. Benedicks, and L. Young \cite{young I}, 1993)}
With respect to the unique physical measure on $\Lambda$, H\'{e}non map $F_{a,b}$ has exponential decay of correlations for each $(a,b) \in S$. 
\qed
\end{theorem}

\begin{remark}
In the proof of this theorem, a important property one should use is the existence of a direction of non-uniform expansion. However, orbits suffer setbacks in expansion when they pass near a localized set of critical points. The decay of correlations takes into account the set of points approaching in a counter-productive way the source of non-expansion. The measure of this set decays exponentially fast to 0.
\end{remark}

\clearpage
\section{The Dynamics of Complex Quadratic Maps under Singular Perturbations $F_{\lambda , m}(z) = z^{2}+ \dfrac{\lambda }{{z} ^{m}}$ and $G_{\beta ,m}(z) = z^{2} + \dfrac{\beta }{\bar{z} ^{m}}$, where $m \in \mathbb{N}$ and $\lambda ,\beta \in \mathbb{C}$.}
\subsection{Preliminary}
\item The goal of studying \textbf {complex dynamics} is to understand the iteration processes of complex analytic functions, which include polynomials, rational maps, entire transcendental maps, and meromorphic functions, on complex plane $\mathbb{C}$, Riemann Sphere $\mathbb{\widehat {C}}$, and even higher dimensional complex plane $\mathbb{C}^{n}$. In complex dynamics, the two most fundamental sets are Julia sets and the Mandelbrot set: the former is geometrically defined as the boundary of the set of the points whose orbits tend to infinity for any fixed map, while the latter is the set of values of the parameter $c$ for which the orbits of $z_{0} = 0$ remains bounded under the complex polynomial $z_{n+1} = z_{n}^{2} + c $ \footnotemark[1].
\renewcommand{\thefootnote}{\fnsymbol{footnote}}
\footnotetext[1] {\emph {\cite{Roeder} provides a different approach to define the Mandelbrot set.}}

The complement of the Julia set is called the Fatou set, whose dynamics, however, is usually relatively tedious (In most cases points in the Fatou set approach an attracting periodic orbits or infinity, although there are some other possibilities). For the details of the history of the complex dynamics, see \cite{Alexander}.
 
\begin{remark}
Three basic classifications of fixed points in dynamical systems are attracting, repelling and neutral points (i.e. the x-values at which $|F'(x)| < 1$, $|F'(x)| > 1$, and $|F'(x)| = 1$, respectively). \\
(1) According to the Contraction Mapping Principle, for all attracting fixed points $z_{0}$, there is an open neighbourhood $U(z_{0}, \epsilon)$ such that $F^{n}(z) \to z_{0}$ as $n \to \infty$, $\forall$ $z \in U(z_{0}, \epsilon)$. \\
(2) For the dynamics of an open neighbourhood of a repelling fixed point $z_{0}$, we can apply the Inverse Function Theorem with the conclusion in (1) to prove that in linear case, $F^{n}(z) \to \infty$ as $n \to \infty$, $\forall$ $z \in U(z_{0}, \epsilon)$. \\
(3) The dynamics nearby a neutral fixed point is much more complicated than the previous two cases. We may obtain attraction, repulsion or other kinds of dynamics in the open neighbourhoods of a neutral fixed point.
\end{remark}

\item In complex dynamics, the quadratic maps: $F_{a}: \mathbb{C}\rightarrow\mathbb{C}$, where $F_{a}(z) = z^{2} + a$, in which $a \in \mathbb{C}$, have been well studied. For this family, there exists only one critical orbit (the orbit of critical point), and we have the following theorem for its escape dichotomy:
\begin{theorem}
For the quadratic map $F_{a}: \mathbb{C}\rightarrow\mathbb{C}$, where $F_{a}(z) = z^{2} + a$, in which $a \in \mathbb{C}$:\\
(1) If the critical orbit remains bounded, then the its Julia Set is connected; \\
(2) Otherwise, the Julia Set is a Cantor Set (also called ``fractal dust") and $F_{a}$ is conjugate on the Julia Set to one-sided shift  of two symbols.
\end{theorem}
For the proof of this theorem and more discussion about the dynamics of this map, see references 
\cite {Devaney III}, \cite {Devaney I} and \cite {Devaney II}.
 
\item In the following sections, we will discuss the dynamics of the complex quadratic map under singular perturbations. Roughly speaking, \textbf {singular perturbation} means introducing poles into the dynamics of a polynomial. It has been shown that singular perturbations can produce rich interesting and elegant results in ODEs, PDEs and dynamical systems.
\item Two main types of singular perturbations in current dynamics research are holomorphic ones, which take the form $F_{\lambda,c,n,m}:\mathbb{C}\rightarrow\mathbb{C}$, where $F_{\lambda,c,n,m}(z) = z^{n} + c + \dfrac{\lambda}{z^{m}}$; and nonholomorphic ones, which are expressed as  $G_{\beta,c,n,m}:\mathbb{C}\rightarrow\mathbb{C}$, where $G_{\beta,c,n,m}(z) = z^{n} + c + \dfrac{\beta}{\bar{z} ^{m}}$ (In both of these two maps, $n,m \in \mathbb{N}$ and $\lambda, \beta, c \in \mathbb{C}$). In this section, we only consider some simple cases for the singular perturbations of complex quadratic maps (i.e. n=2).

\begin{remark}
Besides the holomorphic and nonholomorphic singular perturbations, there are several other types of perturbations in complex dynamics, such as the real (nonholomorphic but nonsingular) perturbation for quadratic families: $F_{\alpha,c}:\mathbb{C}\rightarrow\mathbb{C}$, where $F_{\alpha,c}(z) = z^{2}+ \alpha \bar {z} + c$, in which $c \in \mathbb{C}$ and $\alpha \in \mathbb{R}$. See \cite {Peckham I} for the dynamics of this family.
\end{remark}

\item Before discussing the singular perturbations, we introduce several notations will be used: (1) $J(F)$ is the Julia set of map $F$; (2)$B_{\lambda }(F)$ is the immediate basin of attraction of $\infty $ for $F$, and $\beta _{\lambda }(F)$ is the boundary of $B_{\lambda }(F)$; (3) $K(F) = \mathbb{C} \setminus \bigcup _{n=1} ^{\infty } {F^{-n}(B_{\lambda }(F))}$ is filled Julia set; (4) $T_{\lambda } (F)$, which is called trap door, is the neighborhood of the pole 0 that is mapped onto $B_{\lambda }(F)$ under $F$ but is disjoint from $B_{\lambda }(F)$ (in other words, $T_{\lambda }(F)$ is an open set about the pole 0 that mapped in an ``m to one" fashion onto $B_{\lambda }(F)$ under $F$.

\item In the following sections, we will discuss the dynamics of $z^{2}$ rather than the more general quadratic family ($z^{2} + c$) under both holomorphic and nonholomorphic singular perturbations with different orders (m); i.e. $F_{\lambda , m}(z) = z^{2}+ \dfrac{\lambda }{{z} ^{m}}$ and $G_{\beta ,m}(z) = z^{2} + \dfrac{\beta }{\bar{z} ^{m}}$, where $m \in \mathbb{N}$ and $\lambda ,\beta \in \mathbb{C}$. 

\clearpage
\subsection{Singular Perturbation of Real Quadratic Family when $m=1$}
\begin{figure}[h]
\centering
\centerline{
\includegraphics[scale = 0.4]{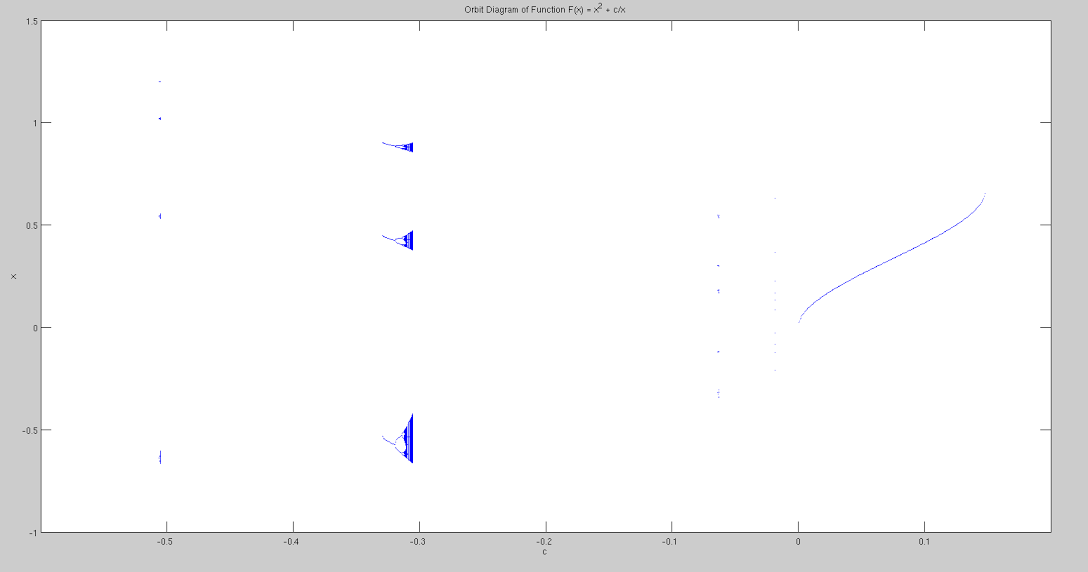}}
\caption{\small The Orbit Diagram of $F_{c}: \mathbb{R}\rightarrow\mathbb{R}$, where $F_{c}(x) = x^{2} + \dfrac{c}{x}$, in which $c \in \mathbb{R}$.} 
\end{figure}

\item Before discussing the complex quadratic families, we firstly study the singular perturbation in the case of the real line $\mathbb{R}$; i.e. of $F_{c}: \mathbb{R}\rightarrow\mathbb{R}$, where $F_{c}(x) = x^{2} + \dfrac{c}{x}$, in which $c \in \mathbb{R}$. To gain an overview of the dynamics for all c, we first observe the \textbf {orbit diagram} of this family (Fig.4) and interpret some interesting dynamics. From the orbit diagram, one can see that the dynamics of $c>0$ and $c<0$ are entirely different. We firstly discuss the dynamics when $c>0$, and then analyze the $c<0$ in the following section. 

\begin{remark}
Orbit diagram shows the asymptotic behaviors of the orbits of critical points for various c-values. It aims to capture the dynamics of a family of maps for many different c-values in one picture\cite{Devaney III}. This helps us to find the attracting periodic orbits of maps, because every attracting periodic orbit attracts a critical point. One should notice that only the points whose orbits are stay bounded will be shown on orbit diagrams. 
\end{remark}

\begin{figure}[h]
\centering
\centerline{
\includegraphics[scale = 0.55]{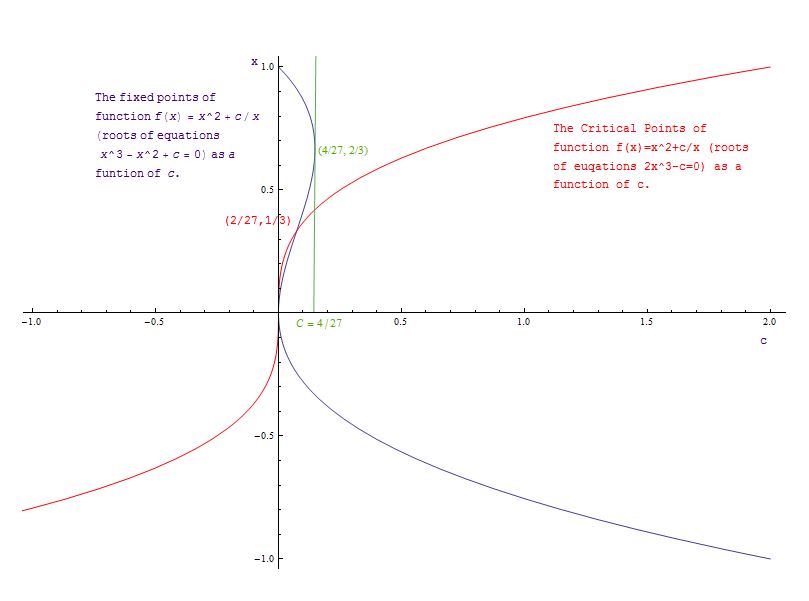}}
\caption{\small Analysis of the Dynamics when $c>0$. In this figure, purple line represents the fixed points as a function of parameter c's; Red line represents the critical points as a function of parameter c's; Green line represents the straight line $c=4/27$, which helps to see the saddle-node point at $(c ,x) = (4/27, 2/3)$.} 
\end{figure}

\begin{remark}
For the quadratic family $F_{c}: \mathbb{R}\rightarrow\mathbb{R}$, where $F_{c}(x) = x^{2} + \dfrac{c}{x}$, in which $c \in \mathbb{R}$, the dynamics at $(c, x) = (0, 0)$ and $(4/27, 2/3)$ are entirely different, although both of them are the two ends of the $S$-shape curve on the orbit diagram. Because $x=0$ is the singular point, its orbit will approach to infinity after only one iteration.  
\end{remark}

\begin{figure}[h]
\centering
\centerline{
\includegraphics[scale=0.3]{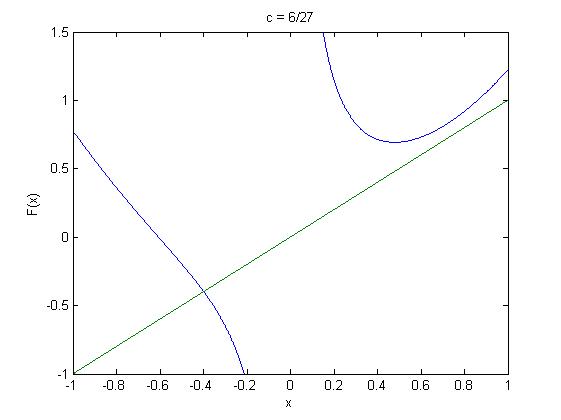}
\includegraphics[scale=0.3]{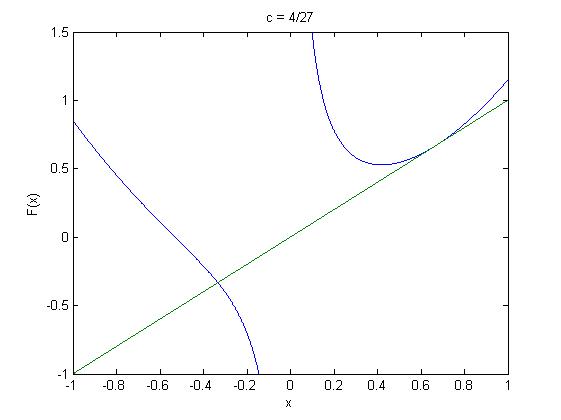}
\includegraphics[scale=0.3]{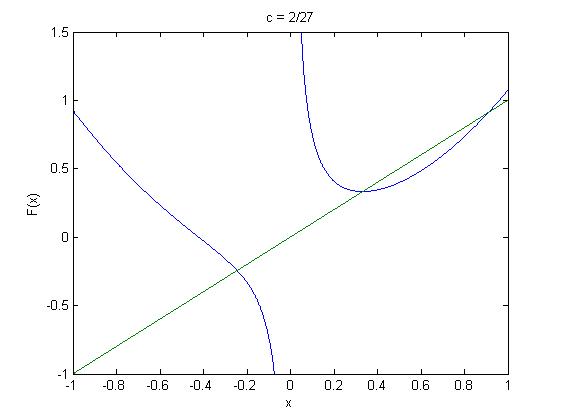}}
\caption{\small A Description of the Saddle-Node Bifurcation - The Graph of $F_{c}(x)$ when $c = \dfrac{6}{27}$, $c = \dfrac{4}{27}$ and $c = \dfrac{2}{27}$, respectively.}
\end{figure}
\item We start analyzing the dynamics in the case of $c>0$ through observing Fig.5. The two curves and straight line represent the fixed points as a function of parameter $c$'s, the critical points as a function of parameter c's, and the straight line $c=4/27$, respectively. They intersect at several points: (0, 0), (2/27, 1/3), and (4/27, 2/3) (we do not discuss (4/27, -1/3) here because the repelling fixed point persists as $c$ is varied). Point (0,0) is where singularity; (2/27, 1/3) is located at center of the left $S$-shape curve, and 1/3 is the superattracting fixed (both attracting fixed and critical) point; point (4/27, 2/3) is where saddle-node bifurcation occurs. More specifically, when $c > 4/27$, there is only one negative fixed point; when $c = 4/27$, there are two fixed points, one is negative while the other is positive; when  $c < 4/27$, there are three fixed points (two of them are positive, while another is negative). Therefore, a saddle-node bifurcation occurs when $c = 4/27$, and the negative fixed point always exists for any nearby value of $c$. For convenience, we denote the fixed points from left to right as $x_{1}$, $x_{2}$ and $x_{3}$. Now we discuss the properties of these fixed points.

\begin{theorem}
For the family $F_{c}: \mathbb{R}\rightarrow\mathbb{R}$, where $F_{c}(x) = x^{2} + \dfrac{c}{x}$, in which $c > 0$ and $c \in \mathbb{R}$ (assume $x \neq 0$), we have\\  
(1) When $c > 4/27$, there exists only one repelling fixed point  $x_{1}$;\\
(2) When $c = 4/27$, there exist two fixed points, $x_{1}$ is repelling while $x_{2}$ is neutral;\\
(3) When $c < 4/27$, there exist three fixed points, $x_{1}$ and $x_{3}$ are repelling while $x_{2}$ is attracting.
\qed
\end{theorem}

\begin{proof}
For a fixed point $x_{0}$ of $F_{c}:\mathbb{R}\rightarrow \mathbb{R}$, where $F_{c}(x) = x^{2} + \dfrac{c}{x}$, we have
\begin{equation*}
F_{c}^{'}(x) = 2x - \dfrac{c}{x^{2}} = 3x - \dfrac{1}{x} ( x^{2} + \dfrac{c}{x}) = 3x - \dfrac{1}{x} F_{c}(x) \implies F_{c}^{'}(x_{0}) = 3x_{0} - \dfrac{1}{x_{0}} F_{c}(x_{0}) = 3x_{0} - 1.
\end{equation*}

Therefore, if $0<x_{0}<2/3$, $x_{0}$ is attracting; if $x_{0}<0$ or $2/3<x_{0}$,  $x_{0}$ is repelling; otherwise,  $x_{0}$ is neutral.\\
(1) When $c > 4/27$, there exists only one fixed point $x_{1}$, then $\forall c \in \mathbb{R}$, $x_{1} < 0$  $\implies$ $x_{1}$ is always repelling $ ( $ the same reason for the repelling fixed points in (2) and (3)$ ) $;\\
(2) When $c = 4/27$, there exist two fixed points $x_{1}$ and $x_{2}$, and $x_{2} = 2/3$ $\implies$ $x_{2}$ is neutral;\\
(3) When $c < 4/27$, there exist three fixed points $x_{1}$, $x_{2}$ and $x_{3}$.  When c increases, $x_{2}$ increases and $x_{3}$ decreases, and they coincide at $c = 4/27$. Theretofore, when $c < 4/27$, we have $0 < x_{2} < 2/3$ and  $2/3 < x_{3}$ $\implies$ $x_{2}$ is attracting and $x_{3}$ is repelling.
\end{proof}

\item Now we consider the dynamics when $c<0$. From Fig.4, we see that several period-doubling route to chaos with different primary periods appear on the left side in the graph. Fig.7 shows the period-doubling route with period three and four on the orbit diagram. To interpret the dynamics when $c<0$, we firstly analyze the dynamics of the period-doubling bifurcation with period three. In the right graph in Fig.8, the curves with three different colors represent $F_{c}(x) = x^{2} + \dfrac{c}{x}$, $F_{c}^{2}(x)$, and $F_{c}^{3}(x)$, respectively. The left graph in Fig.8 shows the point at which period-doubling route occurs ($c=-0.327$). From left to right, three blue curves intersect (are tangent to) the reference line $y=x$ simultaneously, which means that $F_{c}^{3}(x) = x$. When c-value becomes smaller than -0.327, one can see three more periodic points appear (right graph in Fig.8). The rightmost point at which the four curves (including the green reference line $y=x$) with different colors intersect is the fixed point of $F_{c}(x)$. It worths to mention that in Fig.8 there are three intervals between the intersections between the $F_{c}^{3}(x)$ curves and $y=x$ line , and the lengths of these three intervals correspond to the vertical heights of three pieces in the left graph in Fig.7.  

\begin{figure}[h]
\centering
\centerline{
\includegraphics[height=5cm, width=9cm]{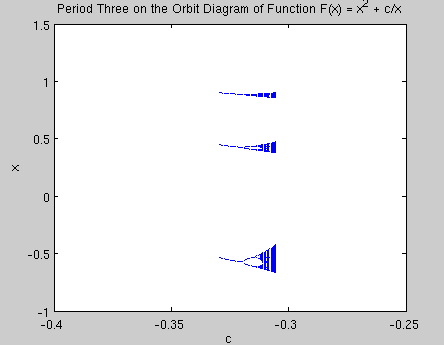}
\includegraphics[height=5cm, width=9cm]{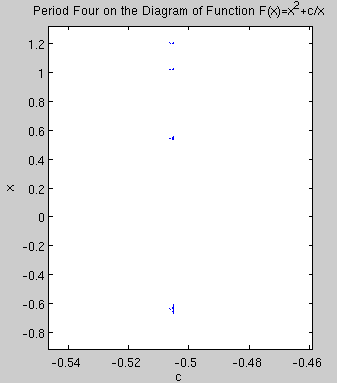}}
\caption{\small Enlargements from Fig.4 - Period-Doubling Routes with Period Three and Four on the Orbit Diagram.}
\end{figure}
 
\begin{figure}[h]
\centering
\centerline{
\includegraphics[height=6cm, width=6.5cm]{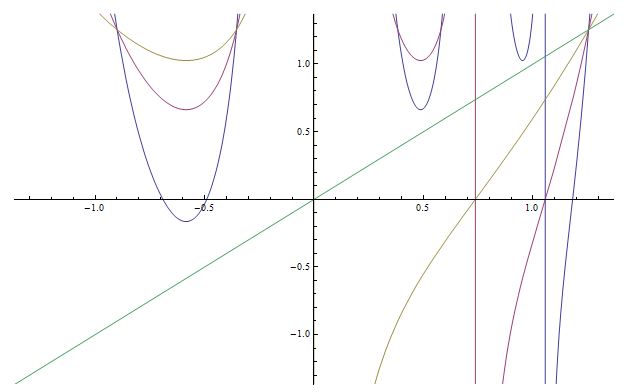}
\includegraphics[height=6cm, width=6.5cm]{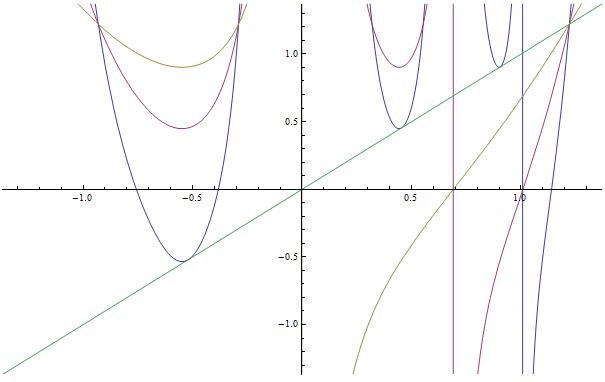}
\includegraphics[height=6cm, width=6.5cm]{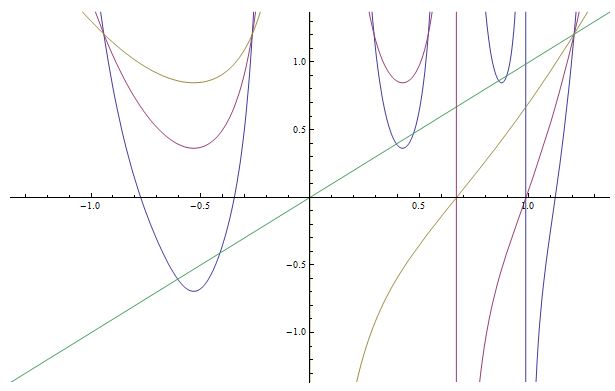}}
\caption{\small The Graphs of $F_{c}(x) = x^{2} + \dfrac{c}{x}$, $F_{c}^{2}(x)$, and $F_{c}^{3}(x)$ when $c<0$. Left: before the saddle-node bifurcation, $c>4/27$; middle: at the (period-three) saddle-node bifurcation, $c=4/27$; right: after the saddle-node bifurcation, $c<4/27$. (Yellow curves represent $F_{c}(x)$, purple curves represent $(F_{c}^{2}(x))$, blue curves represent $(F_{c}^{3}(x))$, and green straight line represents the reference line $y=x$.)}
\end{figure}

\item Now we suggest two conjectures based on the several typical iteration graphs when $c<0$. According to the orbit diagram in Fig.4 and four iteration graphs in Fig.9, we find that for those period-$n$ cycles appear when $c<0$, the integer $n$ keeps increasing when c-value approaches the \textbf {homoclinic} case, in which the critical value (the lowest point on the left branch of $F_{c}(x) = x^{2} + \dfrac{c}{x}$ curves) and the fixed point (the intersection of the reference line $y=x$ and the left branch of $F_{c}(x) = x^{2} + \dfrac{c}{x}$ curves) have the same height. When c-value is less than -0.593, all orbits that start from the critical point will go off to infinity. Therefore, we have the following two conjectures: 

\begin{conjecture}
\item Infinitely many period-$n$ cycles with different $n$-values appear when $c$-value varies from -0.3237 (period-three) to -0.593 (homoclinic).  
\end{conjecture}

\begin{conjecture}
\item Between each pair of the successive period-$n$ cycles, there exist c-value(s) under which the orbit of the critical point approaches to infinity. 
\end{conjecture}

\begin{remark}
\item As we mentioned in the Introduction section, a homoclinic orbit has common forward and backward asymptotic behaviors to the same point. See \cite{Robinson} and \cite{Devaney VI} for more discussion. 
\end{remark}

\begin{figure}[h]
\centering
\centerline{
\includegraphics[height=6cm, width=8cm]{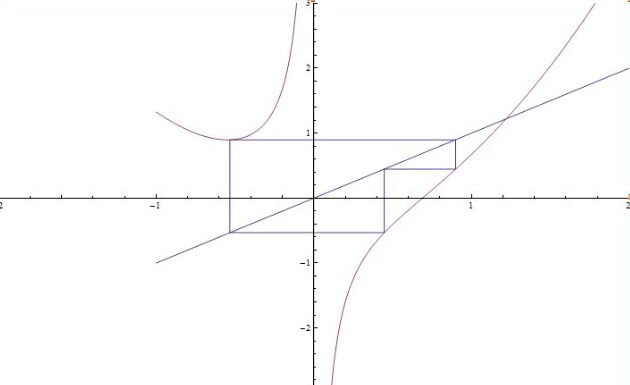}
\includegraphics[height=6cm, width=8cm]{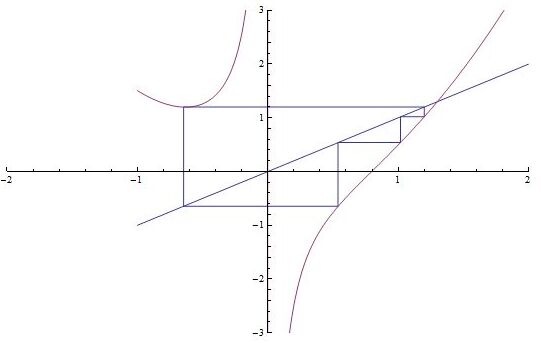}}
\centerline{
\includegraphics[height=6cm, width=8cm]{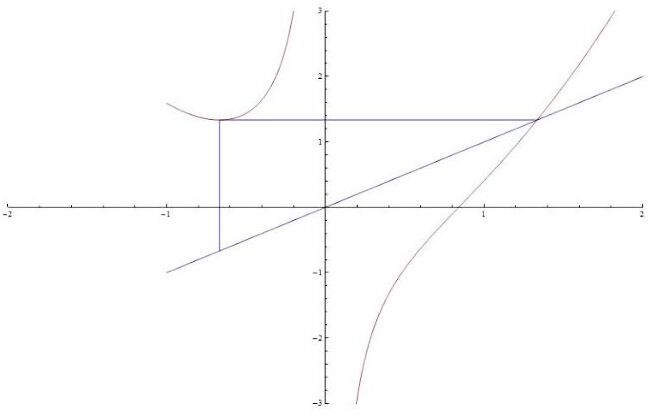}
\includegraphics[height=6cm, width=8cm]{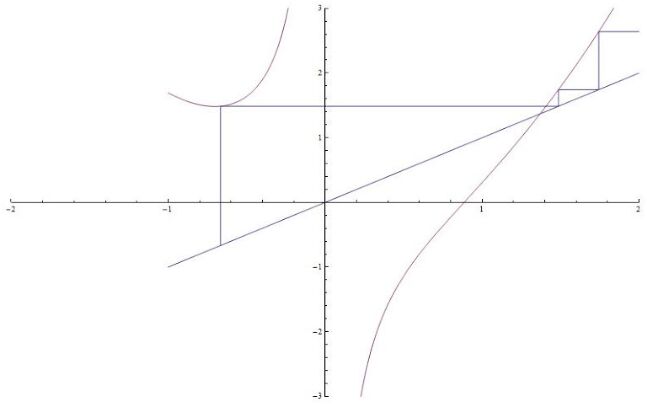}}
\caption{\small Iteration Graphs of Cycles with Different Periods for $F_{c}(x) = x^{2} + \dfrac{1}{x}$. (Upper left: period-three when $c = -0.327$; upper right: period-four when $c = -0.507$; lower left: homoclinic (pre-fixed) when $c = -0.593$; lower right: escape when $c < -0.593$.)}
\end{figure}

\clearpage
\subsection{Singular Perturbations of Complex Quadratic Family when $m=1$}
\item Now we consider the dynamics of $F_{c}: \mathbb{C}\rightarrow\mathbb{C}$, where $F_{c}(z) = z^{2} + \dfrac{c}{z}$, in which $c > 0$ and $c \in \mathbb{R}$. When $c < 4/27$ and $c = 4/27$, the fixed points of this family are the same as the family $F_{a}(z) = z^{2} + a $ referred in \textbf {section 3.1}; however, when $c > 4/27$, the fixed points are different due to the appearance of a pair of complex conjugate fixed points. For convenience, we again denote the fixed points as $z_{1}$, $z_{2}$ and $z_{3}$. The dynamics of this family at the fixed points is stated as follows:

\begin{theorem} 
For the family $F_{c}: \mathbb{C}\rightarrow\mathbb{C}$, where $F_{c}(z) = z^{2} + \dfrac{c}{z}$, in which $c > 0$ and $c \in \mathbb{R}$, we have, we have\\
(1) When $c > 4/27$, there exists one real and two complex repelling fixed point $z_{1}$, $z_{2}$ and $z_{3}$;\\
(2) When $c = 4/27$, there exist two real fixed points, $z_{1}$ is repelling while $z_{2}$ is neutral;\\
(3) When $c < 4/27$, there exist three real fixed points, $z_{1}$ and $z_{3}$ are repelling while $z_{2}$ is attracting.
\qed
\end{theorem}

\begin{proof}
When $c < 4/27$ and $c = 4/27$, the fixed points of $F_{c}:\mathbb{R}\rightarrow \mathbb{R}$, where $F_{c}(x) = x^{2} + \dfrac{c}{x}$, reduces to the case in Theorem 3.2. Therefore, we only need to prove its .fixed points when $c > 4/27$: \\
Assume $z_{0}$ is the fixed point of $F_{c}:\mathbb{R}\rightarrow \mathbb{R}$, where $F_{c}(z) = z^{2} + \dfrac{c}{z}$, then 
\begin{equation*}
z^{2}+\dfrac{c}{z}=z \implies z^{3}-z^{2}+c=0.   
\end{equation*}
According to Vieta's formula, one can easily show $z_{1}+z_{2}+z_{3} = 1$. Since the point $z_{1}$ remains negative for all $c > 0$ and $c \in \mathbb{R}$ and decreases as $c$ increases for $c > 4/27$, we know that $z_{2}+z_{3}$ increases as $c$ increases for $c > 4/27$. Since when $c = 4/27$, where the bifurcation happens, we have $z_{1}=-1/3$, then $z_{2}+z_{3}=2/3$ in this case. Therefore, $z_{2}+z_{3}>4/3$, which implies that $Re(z_{2})>2/3$ and $Re(z_{3})>2/3$ when $c > 4/27$. Meanwhile, similar to the real case in \textbf {section 3.1}, one can show that, for the fixed point $z_{0}$, we have 
\begin{equation*}
F_{c}^{'}(z_{0}) = 3z_{0}-1,
\end{equation*}
and then
\begin{equation*}
|F_{c}^{'}(z_{0})| = |3z_{0}-1| > 1,
\end{equation*}
when $z_{0}=z_{2}$ or $z_{0}=z_{3}$. Hence, all the three fixed points $z_{1}$, $z_{2}$ and $z_{3}$ are repelling. 
\end{proof}

\begin{figure}[h]
\centering
\centerline{
\includegraphics[height=4.5cm, width=4.5cm]{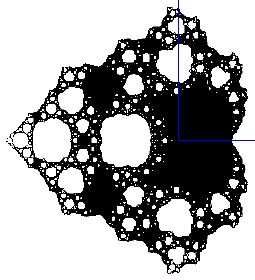}
\includegraphics[height=5cm, width=5cm]{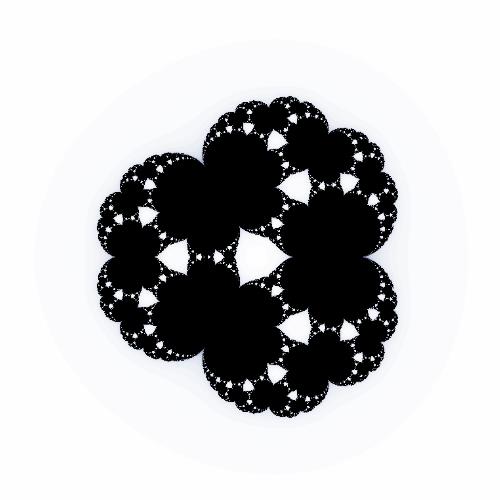}
\includegraphics[height=5cm, width=5cm]{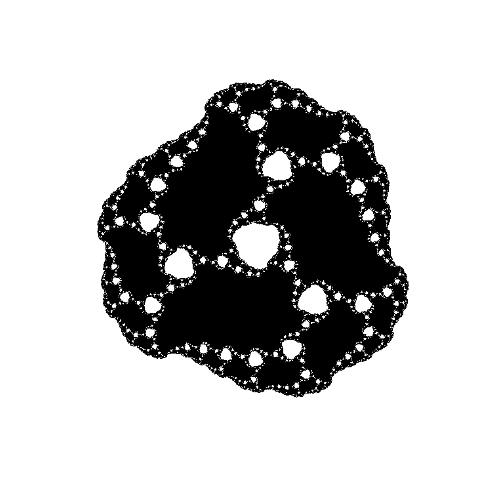}
\includegraphics[height=5cm, width=5cm]{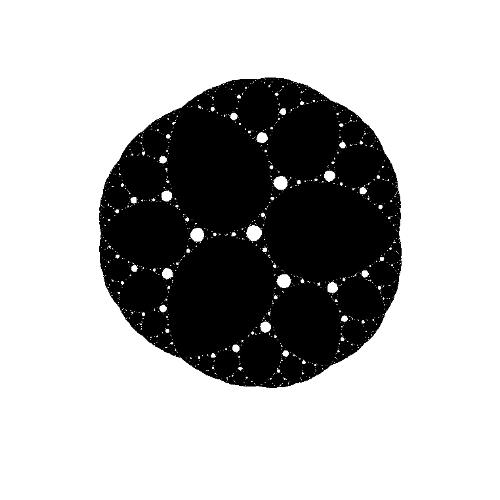}}
\caption{\small From Left to Right: The Escape Figure of (1)the Parameter Plane of the Family $F_{\lambda}(z) = z^{2} + \dfrac{\lambda}{z}$, where $\lambda \in \mathbb{C}$; (2)the Dynamical Plane of $z^{2} + \dfrac{4/27}{z}$; (3)the Dynamical Plane of $z^{2} + \dfrac{0.1+0.1i}{z}$; (4)the Dynamical Plane of $z^{2} + \dfrac{0.01}{z}$.}
\end{figure}

\item Now we discuss the dynamics of a more general case: $F_{\lambda}: \mathbb{C}\rightarrow\mathbb{C}$, where $F_{\lambda}(z) = z^{2} + \dfrac{\lambda}{z}$, in which $\lambda \in \mathbb{C}$. Firstly, we consider its symmetric structure. Let 
\begin{equation*}
\omega = cos(\dfrac{2\pi}{3}) + i\, sin(\dfrac{2\pi}{3}),
\end{equation*}
then
\begin{equation*}
F_{\lambda}(\omega z) = (\omega z)^{2} + \dfrac{\lambda}{\omega z} = {\omega}^{2} (z^{2} + \dfrac{\lambda}{{\omega }^{3} z}) = {\omega}^{2} (z^{2} + \dfrac{\lambda}{z}) = {\omega}^{2}F_{\lambda}(z),
\end{equation*}
and therefore $F_{\lambda}(\omega z)$ and $F_{\lambda}(z)$ are symmetric to each other. Similarly, one can easily show that 
\begin{equation*}
F_{\lambda}({\omega}^{2} z) = ({\omega}^{2} z)^{2} + \dfrac{\lambda}{{\omega}^{2} z} = {\omega} ({\omega}^{3} z^{2} + \dfrac{\lambda}{{\omega }^{3} z}) = {\omega} (z^{2} + \dfrac{\lambda}{z}) = {\omega} F_{\lambda}(z),
\end{equation*}
which implies that $F_{\lambda}({\omega}^{2} z)$ and $F_{\lambda}(z)$ are symmetric as well. Therefore, $F_{\lambda}(z)$, $F_{\lambda}(\omega z)$ and $F_{\lambda}({\omega}^{2} z)$ are symmetric and obey the same dynamics (i.e. either all approach to infinity or all stay bounded).  

\item For this family, there is only one ''dividing ray", that is, the negative real axis $ \text{Arg}\left(\lambda \right) = \pi $. Then, one can show the following convergence theorem in the ``Hausdorff topology" sense \cite{Morabito}:
\begin{theorem}
\emph{(R. Devaney, and M. Morabito \cite{Morabito}, 2004)}
For $F_{\lambda}(z) = z^{2} + \dfrac{\lambda}{z}$, in which $\lambda \in \mathbb{C}$, its Julia set $J(F_{\lambda})$ converges as a set to the closed unit disk as $\lambda$ approaches to 0 along the dividing ray $\text{Arg}\left( \lambda \right) = \pi $ in its parameter plane.
\qed
\end{theorem}

\begin{proof}
Let $B_{\epsilon }(z)$ be a ball of radius $\epsilon $ centred at the point $z \in \mathbb{C}$, and let $\mathbb{D}$ be the closed unit disk in the complex plane $\mathbb{C}$. Then for any given $\epsilon > 0$ , if $|\lambda| \leq \epsilon$ and $|z| > 1 + \epsilon$, then we have
\begin{equation*}
|z|^{3} - |z|^{2} = |z|^{2}(|z| - 1) > |z|^{2} \epsilon > \epsilon \geq |\lambda|.
\end{equation*}
Then,
\begin{equation*}
|F_{\lambda}(z)| = |z^{2} + \dfrac{\lambda}{z}| \geq |z|^{2} - \dfrac{|\lambda|}{|z|} = \dfrac{1}{|z|} \left( |z|^{3} - \lambda \right) \geq \dfrac{1}{|z|} |z|^{2} = |z| > 1 + \epsilon. 
\end{equation*}
Therefore. for each $z$ satisfying $|z| > 1 + \epsilon $, we have $z \notin J(F_{\lambda})$. In other words, for each $z \notin \mathbb{D}$, we have $z \notin J(F_{\lambda})$.
\item As claimed before, the dividing ray of this complex map is exactly the negative real line $\mathbb{R}^{-}$. Let $\omega ^{3} = 1$, then $F_{\lambda}(\omega z) = \omega ^{2}F_{\lambda}(z)$, which implies that the Julia set $J(F_{\lambda})$ is symmetric under $z \mapsto \omega z$. Meanwhile, since $F_{\lambda}(\bar z) = \overline{F_{\lambda}(z)}$ for the parameter $\lambda \in \mathbb{R}^{-}$, which implies that the Julia set is symmetric under complex conjugation. Since that $z \notin J(F_{\lambda})$ for each $z \notin \mathbb{D}$ have been proved, what should be proved for the theorem is that for any $z \in \mathbb{D}$, $J(F_{\lambda}) \cap B_{\epsilon }(z) \neq \emptyset $. We will prove this by contradiction as follows.
\item Let us assume that the Julia set $J(F_{\lambda})$ of this family does not converge to the closed unit disk $\mathbb{D}$ as the parameter $\lambda$ approaches 0 along the negative real line $\mathbb{R}^{-}$; meanwhile, there is a sequence $\left( z_{j} \right) _{0}^{\infty } \subset \mathbb{D}$ such that $J(F_{\lambda}) \cap B_{\epsilon }(z) = \emptyset $. Since $\mathbb{D}$ is bounded and closed, then it is a compact region by the Heine-Borel Theorem. Then there exists a subsequence $( a_{k}) _{0}^{\infty }$, which consists of the points in the sequence $\left( z_{j} \right) _{0}^{\infty }$ that converges to some point $p \in \mathbb{D}$, such that $J(F_{a_{k}}) \cap B_{\epsilon }(p) = \emptyset $ for some sufficiently large $k$. 
\item Suppose $|\lambda|$ is small, let $|\lambda| < 1/27$. Then if z is on the circle of radius $|\lambda|^{1/3}$ centered at 0, and denote this circle as $\mathbb{E}$. Then the following inequality holds:
\begin{equation*}
|F_{\lambda }(z)| \leq |z|^{2} + \dfrac{|\lambda|}{|z|} = |\lambda|^{2/3} + \dfrac{|\lambda|}{|\lambda|^{1/3}} = 2|\lambda|^{2/3} < |\lambda|^{1/3} = |z|.
\end{equation*}
Therefore, for small $|\lambda |$, the circle $\mathbb{E}$ is strictly mapped inside itself, which implies that the boundary of trap door $T_{\lambda}$ of this map $|F_{\lambda }|$, which is denoted as $\partial T_{\lambda }$, lies in the circle $\mathbb{E}$ for all such $\lambda $. And thus $\lim_{|\lambda | \to 0} \partial T_{\lambda } = 0$. It follows that for small $|\lambda |$, there are some points in the Julia set of this map arbitrarily close to the origin and hence lies inside the ball $B_{\epsilon} (0)$. Thus, if we assume $p = 0 $, then $J(F_{\lambda }) \cap B_{\epsilon }(p) \neq \emptyset $, which contradicts the conclusion under the assumption that the Julia set $J(F_{\lambda})$ of this map does not converge to the closed unit disk  $\mathbb{D}$ as the complex parameter $\lambda $ approaches the origin along the negative real line $\mathbb{R}^{-}$. Therefore, the subsequence $( a_{k}) _{0}^{\infty }$ cannot converges to 0, and hence $p \neq 0$.
\item Now we consider a circle centered at 0 with radius $|p| > 0$, and denote it as $\mathbb{G}$. Then $\mathbb{G} \cap B_{\epsilon} (p) \neq \emptyset $, and denote the minor arc between their two intersection points as $\zeta $ and its length as $l$. Now we choose $m \in \mathbb{N}^{+}$ such that $\left( 2^{m} \right) l > 2\pi $. Therefore, if $z$ lies in a circle centered at the origin with radius $\dfrac{|p|}{2}$, then  for sufficiently small $|\lambda| > 0$, $|F_{\lambda}^{j}(z) - z^{2^{j}}| \rightarrow 0$.  
\item Thus, the argument of the curve $F_{\lambda} (\zeta )$ increases by $2\pi $ approximately, and therefore  it wraps around the origin at least once. Since there is an three-fold symmetry in the dynamical plane, then the curve $\zeta $ intersect all these three lines $\ell _{1}$, $\ell _{2}$ and $\ell _{3}$. As shown before, $z \mapsto \bar {z}$ since $F_{\lambda}(\bar {z}) = \overline{F_{\lambda}(z)}$; then we know that the $k-th$ fold iteration of $B_{\epsilon}(p)$ contains an annulus that lies in the Fatou set and surrounds the origin. Let $Q$ be the component of the Fatou set that contains the curve $\zeta $, then $Q$ is mapped onto a component of the Fatou set that is periodic (denote this new component as $U$) by the No-Wandering Domain Theorem. However, note that the set $H = \left\{ z: z \in \bigcup_{i=1}^{3} \left( Q \cap \ell _{i} \right) \right\} $ remains on these lines for all iterations. Therefore, $U$ cannot be a basin of attraction of a finite cycle, a Siegel disk, or a Herman ring. Thus, we can conclude that $U = B_{\lambda}(F_{\lambda})$, where $B_{\lambda}(F_{\lambda})$ is the immediate basin of the $\infty $. It follows that $f^{-1}(U) = T_{\lambda}(F_{\lambda}(z))$, where $T_{\lambda}$ is the trap door; and then $W = f^{-2}(U)$ is a Fatou component that contains an annulus that surrounds the origin. Since the trap door $T_{\lambda}$ is a disk and $V$ is not simply connected, then $V$ contains at least one critical point of $F_{\lambda}$. It follows that $V$ contains all critical points take this form, $\left( \lambda/2 \right)^{1/3}$, by symmetry. Then, by the Riemann-Hurwitz Theorem, $V$ is an annulus that is mapped 2 to 1 onto $T_{\lambda}$. Let $O$ be the open annulus lies between $B_{\lambda}$ and $T_{\lambda}$. Then $O$ is separated by $V$ into two annuli, an inner one denoted as $O_{1}$ and an outer one denoted as $O_{2}$. Then $O_{1}$ is mapped onto $O$ under $F_{\lambda}$ in a one-to-one pattern, while $O_{2}$ is mapped onto $O$ under $F_{\lambda}$ in an $n$ to one pattern. It follows that mod $O$ = mod $O_{1}$. Since the inner boundaries of $O$ and $O_{1}$ overlap, then $O_{2}$ cannot exist. Therefore, we obtain a contradiction. Therefore, for any $z \in \mathbb{D}$, $B_{\epsilon }(z) \cap J({F_{\lambda}}) \neq \emptyset $. Now the theorem is proved.    
\end{proof}

\begin{remark}
This remark is about the No-Wandering-Domain Theorem referred in the above proof:
\begin{theorem}
\emph{(D. Sullivan \cite{Sullivan}, 1985)} 
\item Let $f: \mathbb{\widehat{C}}\rightarrow\mathbb{\widehat{C}}$ be a rational map of degree $deg(f) \geq 2$, then $f$ does not have a wandering domain. 
\qed
\end{theorem}
This theorem can be stated as following alternative version: every component $U$ of the Fatou set of this rational map $f$ is eventually periodic; that is, there exist $m > n > 0$, where $m,n \in \mathbb{N}$ such that $f^{m}(U) = f^{n}(U)$. This theorem is first proved by D. Sullivan
 \cite{Sullivan}. And for more discussion about the wandering domain in dynamical systems, see reference \cite{Eckmann}.  
\end{remark}

\begin{remark} 
We briefly discuss the Siegel disk and Herman ring in this remark:
\item Both a Siegel disk and a Herman ring are two types of components of Fatou set. The Fatou component is defined as the maximum connected open subset of the Fatou set. Let $f(z): \mathbb{{C}}\rightarrow\mathbb{\widehat{C}}$ be a holomorphic (or entire) or meromorphic function, and suppose that $V$ is an $n-$periodic Fatou component., then the classifications of the Fatou components are as follows, and one and only one of them will occur:(1)Attracting basin: If for all $z \in V$, $\lim_{k \to \infty } \left( f^{n}\right) ^{k}(z) = q$, where $q$ is an $n-$period attracting point in $V$, then $V$ is an attracting basin; (2)Parabolic basin: If for all $z \in V$, there exists $s \in \partial V$, where $s$ is a rationally indifferent $n-$period point, such that $\lim_{k \to \infty } \left( f^{n}\right) ^{k}(z) = s$, then $V$ is a parabolic basin; (3)Siegel disk: If there exists an analytic homeomorphism $\varphi: V \rightarrow \mathbb{D}$, where $\mathbb{D}$ is a closed unit disk, such that $\varphi \circ f^{k} \circ \varphi ^{-1} (z) = e^{2i\pi \alpha }z$ for some $\alpha \in \mathbb{R} \setminus \mathbb{Q}$ (thus Siegel disks are simply connected by definition); (4)Herman ring: If there exists an analytic homeomorphism $\varphi: V \rightarrow S $, where $S = \left\{ z: 1 < |z| < r \right\} $ for some $r > 1$, such that $\varphi \circ f^{k} \circ \varphi ^{-1} (z) = e^{2i\pi \alpha }z$ for some $\alpha \in \mathbb{R} \setminus \mathbb{Q}$; (5)Baker domain: If for all $z \in V$, $\lim_{k \to \infty } \left( f^{n}\right) ^{k}(z) = \infty $, then $V$ is baker domain. However, note that the case (5) only exists when $f(z)$ is a transcendental function; for polynomials and rational functions, there are only four possibilities (1)$\sim $(4).  
\end{remark}

\begin{remark}
For a more general case, $F_{\lambda}(z) = z^{n} + \dfrac{\lambda}{z}$, in which $\lambda \in \mathbb{C}$, the dividing rays are given by 
\begin{equation*}
Arg \left( \lambda \right) = \dfrac{(2k+1)\pi }{n-1},
\end{equation*}
where $k \in \mathbb{N}$ and $0 \leq k \leq n-1$. In this case, the convergence theorem of Julia set takes a more general form, that is, 
For $F_{\lambda}(z) = z^{n} + \dfrac{\lambda}{z}$, in which $\lambda \in \mathbb{C}$, the Julia set $J(F_{\lambda})$ converges as a set to the closed unit disk as $\lambda$ approaches to 0 along each of the dividing rays $\text{Arg}\left(\lambda \right) = \pi $ in the parameter plane.
\qed
\end{remark}

\item In the previous parts, our discussion mainly focuses on the bounded orbits of this family. Now we consider the points whose orbits approach to infinity. 
\begin{theorem}
For $F_{\lambda}: \mathbb{C}\rightarrow\mathbb{C}$, where $F_{\lambda}(z) = z^{2} + \dfrac{\lambda}{z}$, in which $\lambda \in \mathbb{C}$, the orbit of a point z which satisfies $|z| > 1 + |\lambda|$ approaches to infinity.
\qed 
\end{theorem} 
\begin{proof}
Let $\epsilon > 0$ and $|z| > 1 + |\lambda|$, then $z^{2} > 1$ and $|z| - 1 > |\lambda|$, which implies that $|z^{2}|(|z| - 1 )> |\lambda|$. Therefore, $ 1 + |\lambda| < |z| < |z^{2}| - |\dfrac{\lambda}{z}| < |F_{\lambda}(z)| < |F_{\lambda}^{2}(z)|$, and the sequence $\{ |F_{\lambda}^{n}(z)| \} _{n}$ is monotonically increasing. Assume that this sequence approaches to a finite limit when $n$ approaches to $\infty $, and denote this limit as $a$. Then the orbit of any $|z| > 1 + |\lambda|$ is bounded by a circle at the origin with radius $a$. Since this circle is compact, then there exists one limit point $z_{limit}$ on the circle ($|z_{limit}| = a$) for $\{ |F_{\lambda}^{n}(z)| \} _{n}$, which implies that $|F_{\lambda}(z_{limit})| \leq z_{limit}$. However, we had shown that $|z| < |F_{\lambda}(z)|$ for all $|z| > 1 + |\lambda|$; therefore, this is a contradiction. 
\end{proof}

\item Now we can summarize the escape theorem for this family on the complex plane:
\begin{theorem}
\emph{(R. Devaney, and M. Morabito \cite{Morabito}, 2004)}
For the family $F_{\lambda}: \mathbb{C}\rightarrow\mathbb{C}$, where $F_{\lambda}(z) = z^{2} + \dfrac{\lambda}{z}$, $\lambda \in \mathbb{C}$, and let $C_{0}$ be a critical point of $F_{\lambda}(z)$ then we have:\\
(1)If one and hence all $C_{0} \in B_{\lambda }$, then $J$ is a Cantor set; \\
(2)If one and hence all $C_{0} \notin B_{\lambda }$ but $C_{0} \in T_{\lambda }$ , then $J$ is a Cantor set of simple closed curves;\\ 
(3)If all $C_{0}$ lie in preimages of $T_{\lambda }$ under $F_{\lambda}^{j}$ for some $j > 0$, then $J$ is an S-Curve and hence is homeomorphic to the Sierpi\'{n}ski Curve \footnotemark[1] .
\qed
\end{theorem} 
\renewcommand{\thefootnote}{\fnsymbol{footnote}}
\footnotetext[1] {\emph {The definitions of S-curve and Sierpi\'{n}ski curve will be given in \textbf {Section 3.4}}.}

\clearpage
\subsection{Singular Perturbations of Quadratic Family when $m=2$} 
\item With a simple appearance, it is surprising that the dynamics of this map is the most complicated dynamics in the family $F_{\lambda}(z) = z^{2} + \dfrac{\lambda}{z^{2}}$, in which $\lambda \in \mathbb{C}$. Firstly, let us summarize several important points for the dynamics of this map: (1)One and Only One Pole: $P_{0} = 0$; (2)Four Prepoles: $P_{p} = (-\lambda )^{1/4}$ since ${P_{p}}^{2} + \dfrac{\lambda }{{P_{p}}^{2}} = 0$; (3)Critical Points: $C_{0} = {\lambda }^{1/4}$, the pole 0, and the super-attracting fixed point $\infty $; and the union of the orbits of these critical points is named as \textbf {critical orbit}. 
\begin{remark}
\item 
In \cite{Devaney IV} R.Devaney summarized the three main reasons that make $n = 2$ to be the most complicated case in the family 
$F_{\lambda}(z) = z^{n} + \dfrac{\lambda}{z^{n}}$: (1) There is always a MuMullen domain (whose definition will be given later) around the origin in the parameter plane when $n > 2$, while such a structure does not exist when $n = 2$; (2) The McMullen domain is surrounded by infinitely many \textbf {Mandelpinski necklaces} (the disjoint simple closed curves surround the McMullen domain), while there is none of these structures around 0 in the parameter plane; (3) The Julia set for the map when $n = 2$ converges to a closed unit disk when $\lambda $ approaches to the origin, while the Julia set for the map when $n > 2$ is always Cantor set of simple closed curves.  
\end{remark}

\begin{remark}
\item 
The orbits of the four critical points degenerate to one after two iterations $F_{\lambda }^{2}(C_{0}) = {C_{0}}^{2} + \dfrac{\lambda }{C_{0}^{2}}$, and  all of them are on the circle of radius $|\lambda |^{1/4}$. 
\end{remark}

\item Before further discussion, we introduce several related concepts and theorems. First thing is the \textbf{S-curve}, which is defined as a plane locally connected one-dimensional continuum S such that the boundary of each complementary domain of S is a simple closed curve and any two of these complementary domain boundaries are disjoint \cite{Whyburn}. The other object we need to introduce is \textbf{Sierpi\'{n}ski carpet fractal}, which is constructed as follows \cite{Devaney VII} \cite{Devaney VIII}: (1) Start with a unit square in the plane and divide it into  nine subsquares; (2)Remove the open middle square and leave the other eight closed squares; (3) For the eight closed squares obtained in the last step, repeat the previous two steps, which will leaves 64 smaller squares; (4) Repeat this process infinitely many times, then the Sierpi\'{n}ski Carpet Fractal is constructed. A \textbf{Sierpi\'{n}ski curve} is a planar set that is compact, connected, nowhere dense, locally connected, and any two complementary domains are bounded by mutually disjoint simple closed curves \cite{Look} (in other words, a Sierpi\'{n}ski Curve is a planar set that is homeomorphic to the Sierpi\'{n}ski carpet).   

\item A Sierpi\'{n}ski curve possesses rich topology, and it is called as ``universal" planar sets due to its strong topological property stated in the following theorem \cite{Whyburn}:

\begin{theorem}
\emph{(Whyburn \cite{Whyburn}, 1958)}
Any two S-Curves are homeomorphic, and every S-Curve is homeomorphic with the Sierpi\'{n}ski Curve. 
\qed
\end{theorem}

\begin{figure}[h]
\centering
\centerline{
\includegraphics[scale=0.4]{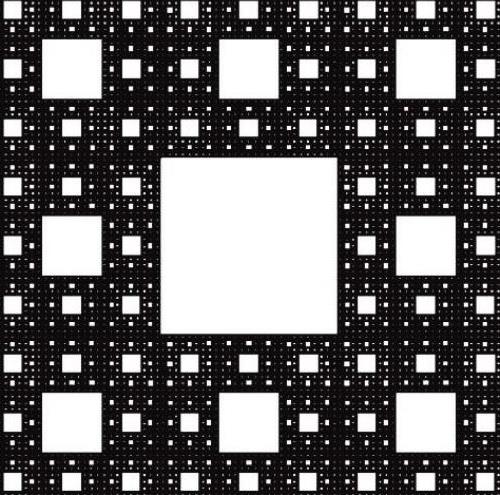}
\includegraphics[scale=0.4]{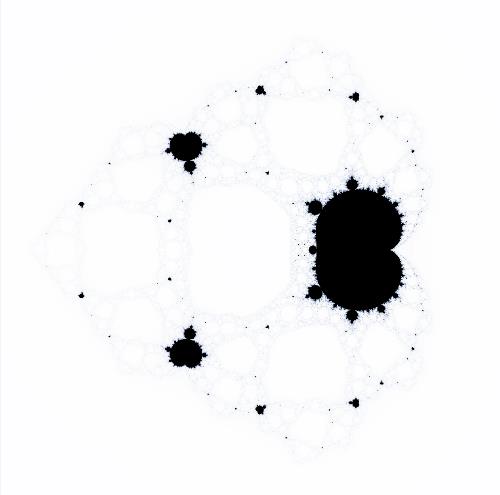}
\includegraphics[scale=0.4]{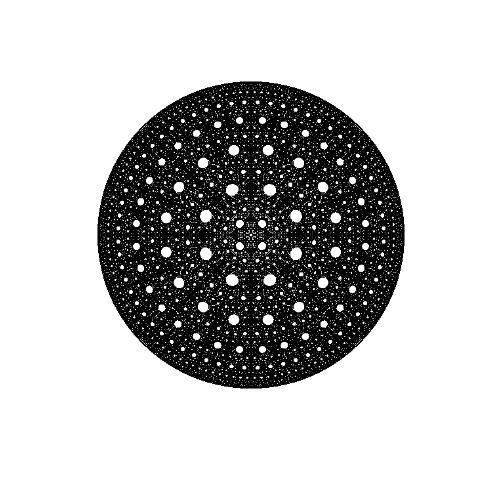}}
\centerline{
\includegraphics[scale=0.4]{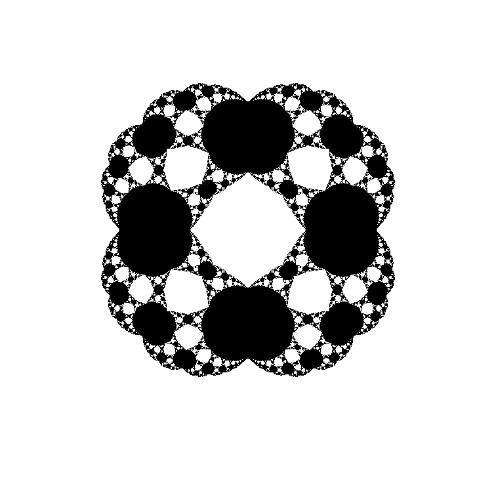}
\includegraphics[scale=0.4]{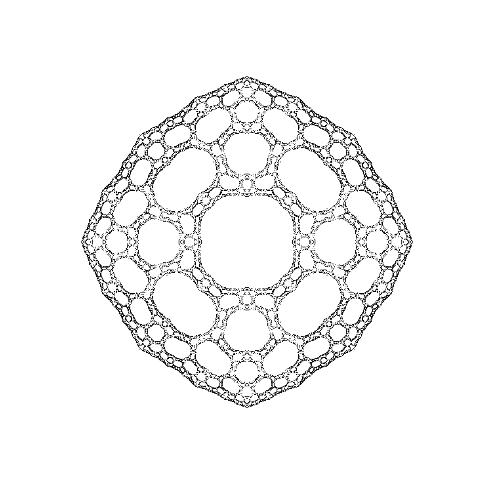}
\includegraphics[scale=0.4]{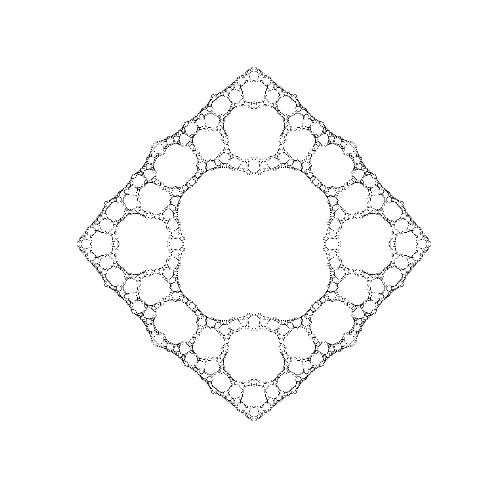}}
\caption{\small (1)Upper Left: The Sierpi\'{n}ski Carpet Fractal; (2)Upper Middle: The Escape Figure of the Parameter Plane of $z^{2} + \dfrac{\lambda}{z^{2}}$; (3)Upper Right: The Julia Set of $z^{2} + \dfrac{0.0001}{z^{2}}$; (4)Lower Left: The Julia Set of $z^{2} + \dfrac{0.1}{z^{2}}$; (5)Lower Middle: The Julia Set of $z^{2} + \dfrac{-0.1}{z^{2}}$ ; (6)Lower Right:  The Julia Set of $z^{2} + \dfrac{-0.25}{z^{2}}$.}
\end{figure}

\item The escape figure of the parameter plane of $F_{\lambda }(z) = z^{2} + \dfrac{\lambda}{z^{2}}$ when $\lambda \in \mathbb{C}$ is shown in Fig.11-(2). The critical orbit for the parameter values in the coloured regions stays bounded, and the Julia set for the parameter values in these regions is connected. The white region represents the parameter values for which the critical orbit escapes to $\infty $. and there are two different dynamics correspond to the parameters in the white regions: (1)The small region in the center of the parameter plane is called \textbf{McMullen domain}, the Julia set for the parameter values in this region is a Cantor set of simple closed curves; (2)For the parameters in other white regions, the Julia set is a Sierpi\'{n}ski Curve, and these regions are called \textbf {Sierpi\'{n}ski holes}. These definitions are the same for the more general family $F_{\lambda}: \mathbb{C}\rightarrow\mathbb{C}$, where $F_{\lambda}(z) = z^{n} + \dfrac{\lambda}{z^{n}}$, in which $\lambda \in \mathbb{C}$, which we will briefly discuss at the end of this section.

\begin{theorem}\footnotemark[1]
\emph{(R. Devaney \cite{Devaney IV}, 2012)}
For the family $F_{\lambda }(z) = z^{2} + \dfrac{\lambda }{z^{2}}$, in which $\lambda \in \mathbb{C}$: \\ 
(1)If one and hence all $C_{0} \in B_{\lambda }$, then $J$ is a Cantor set; \\
(2)If one and hence all $C_{0} \notin B_{\lambda }$ but $C_{0} \in T_{\lambda }$ , then $J$ is a Cantor set of simple closed curves;\\ 
(3)If $C_{0} \notin B_{\lambda }$ but $C_{0} \in {\bigcup}_{i=1}^{\infty} F^{-i}(B_{\lambda })$ , then $J$ is an S-Curve and hence is homeomorphic to the Sierpi\'{n}ski Curve. 
\qed
\end{theorem}

\renewcommand{\thefootnote}{\fnsymbol{footnote}}
\footnotetext[1]{\emph {More discussion about $F_{\lambda }(z) = z^{2} + \dfrac{\lambda }{z^{2}}$, in which $\lambda \in \mathbb{R}$, is in \textbf {Section 3.5}}.}

\begin{remark}
\item According to the the theorem that the any two S-Curves are homeomorphic, we know that for this family, any two Julia sets corresponding to an eventually escaping critical orbit are homeomorphic.
\end{remark}

\item Now we review some results for a more general family $F_{\lambda}: \mathbb{C}\rightarrow\mathbb{C}$, where $F_{\lambda}(z) = z^{n} + \dfrac{\lambda}{z^{n}}$, in which $n \in \mathbb{N}$ and $n \geq 3$, $\lambda \in \mathbb{C}$. Similar to the case $n = 2$, for this function in the families when $n \geq 3$, there are $2n$ critical points $C_{0} = {\lambda }^{1/2n} $ besides 0 and $\infty $, and $2n$ prepoles given by $P_{p} = (-\lambda )^{1/2n}$. The critical points and prepoles are symmetrically arranged due to the following equality holds for any primitive $2n$-th root $\omega $ of unity; i.e. $\omega ^{2n}=1$:
\begin{equation*}
F_{\lambda }(\omega z) = {\omega }^{n} F_{\lambda } (z) = - F_{\lambda } (z).
\end{equation*}   
According to this equality, we can show that $F_{\bar \lambda } (\bar z)= \overline {F_{\lambda } (z)}$, which implies that $J(F_{\lambda })$ is homeomorphic to $J(F_{\bar {\lambda }})$. This allows us to  simplify our discussion by restricting to the case where $Im(\lambda ) \geq 0$. 

\item However, the escape theorem for the cases $n \geq 3$, which is stated in the following theorem, is different from that of $n = 2$:
\begin{theorem}
\emph{(R. Devaney \cite{Devaney IV}, 2012)}
For the family $F_{\lambda}: \mathbb{C}\rightarrow\mathbb{C}$, where $F_{\lambda}(z) = z^{n} + \dfrac{\lambda}{z^{n}}$, in which $n \in \mathbb{N}$ and $n \geq 3$, $\lambda \in \mathbb{C}$: \\ 
(1)If one and hence all $C_{0} \in B_{\lambda }$, then $J$ is a Contor set; \\
(2)If one and hence all $C_{0} \notin B_{\lambda }$ but $C_{0} \in T_{\lambda }$ , then $J$ is a Cantor set of simple closed curves;\\ 
(3)If $C_{0} \notin B_{\lambda }$ and $C_{0} \notin T_{\lambda }$, then $J$ is a connected set;\\
(4)If $C_{0} \notin B_{\lambda }$ and $C_{0} \notin T_{\lambda }$, but $C_{0} \in \left( {\bigcup}_{i=2}^{\infty} F^{-i}(B_{\lambda })\right) \setminus T_{\lambda } $ , then $J$ is an S-Curve and hence is homeomorphic to the Sierpi\'{n}ski Curve.  
\qed
\end{theorem}
In this theorem, we note that the third case does not appear in the theorem for $n = 2$ (actually not true for $n = 1$ either). Besides, there are some other differences between the dynamics of the cases when $n = 2$ and $n > 2$; two typical ones are (1) there exits a McMullen domain whenever $n > 2$; (2) the Julia set does not converge to the unit disk as $\lambda$ approaches $\infty $ \cite{Morabito}.

\begin{remark}
\item Actually, this theorem is also true for a more general family $F_{\lambda}: \mathbb{C}\rightarrow\mathbb{C}$, where $F_{\lambda}(z) = z^{n} + \dfrac{\lambda}{z^{d}}$, in which $n,d \in \mathbb{N}$, $n,d \geq 2$ (but n,d are not both equal to 2), $\lambda \in \mathbb{C}$. See reference \cite {Devaney IV} for more theorems about these more general case.  
\end{remark}

\begin{figure}[h]
\centering
\centerline{
\includegraphics[scale=0.4]{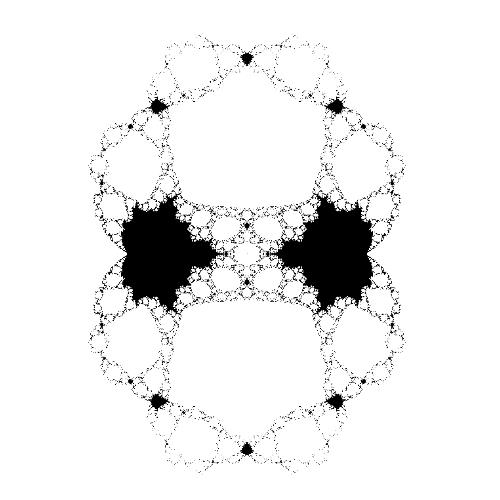}
\includegraphics[scale=0.4]{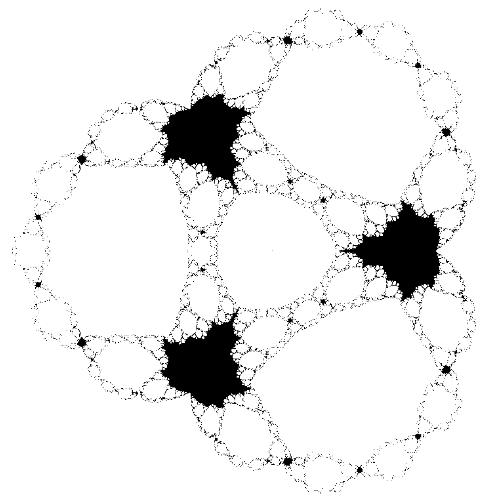}
\includegraphics[scale=0.4]{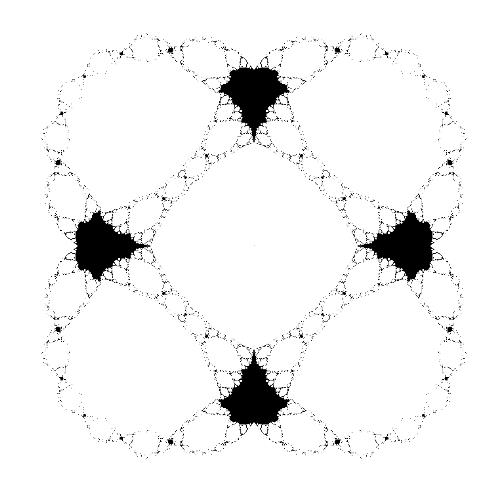}}
\centerline{
\includegraphics[scale=0.4]{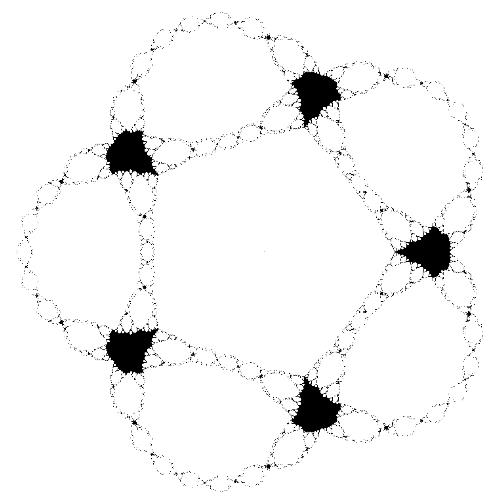}
\includegraphics[scale=0.4]{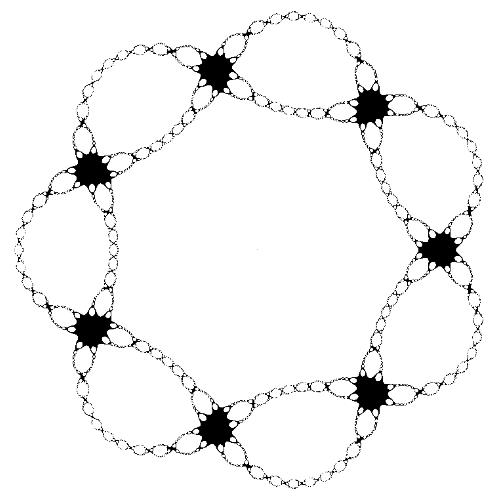}
\includegraphics[scale=0.4]{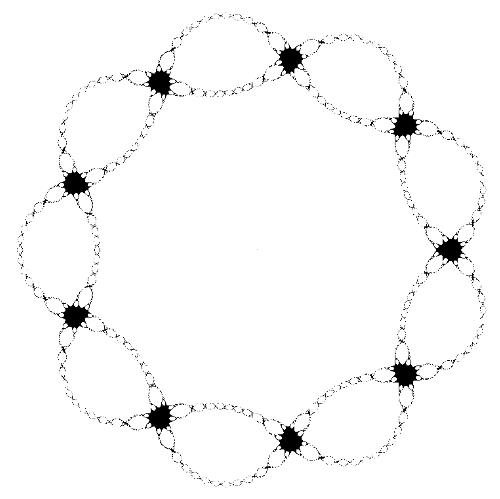}}
\caption{\small The Escape Figures of the Parameter Planes of Several Functions in the Family $F_{\lambda}(z) = z^{n} + \dfrac{\lambda}{z^{n}}$: (1)Upper Left: $z^{3} + \dfrac{\lambda}{z^{3}}$; (2)Upper Middle: $z^{4} + \dfrac{\lambda}{z^{4}}$; (3)Upper Right: $z^{5} + \dfrac{\lambda}{z^{5}}$; (4)Lower Left: $z^{6} + \dfrac{\lambda}{z^{6}}$; (5)Lower Middle: $z^{8} + \dfrac{\lambda}{z^{8}}$; (6)Lower Right: $z^{10} + \dfrac{\lambda}{z^{10}}$.}
\end{figure}

\item Now we consider an interesting result about the convergence of the Julia set $J(F_{\lambda })$:
\begin{theorem}
Let $\epsilon > 0$ and let $B_{\epsilon }(z)$ denote a disk centred at $z$ with radius $\epsilon $. Then there exists $\gamma > 0$ such that, for any $\lambda $ satisfying $0 < |\lambda | \leq \gamma $, $J(F_{\lambda }) \cap B_{\epsilon }(z) \neq \emptyset $ for all $z \in \mathbb{D}$, where $\mathbb{D}$ is the unit disk.  
\qed
\end{theorem}

\begin{proof}
Let us prove this theorem by contradiction. We assume for any given $\epsilon > 0$, there exist a sequence of parameters $\{\lambda _{i} \}_{i = 0}^{\infty }$ which converges to 0, and a sequence $( z_{i} ) _{i = 0}^{\infty }$ in which $z_{i} \in \mathbb{D}$ for all $i \in \mathbb{N}$, such that $J(F_{\lambda _{i}}) \cap B_{\epsilon }(z_{i}) = \emptyset $ for all $i \in \mathbb{N}$. Since the unit disk $\mathbb{D}$ is compact, then there exist a subsequence $(z_{j}) _{j = 0}^{\infty }$ that converges to some point $p \in \mathbb{D}$. Then for each parameter in the corresponding subsequence, $J(F_{\lambda _{j}}) \cap B_{\epsilon }(p) = \emptyset $. Let $\mathbb{G}$ denote a circle centred at 0 with radius $p$, then $\mathbb{G} \cap B_{\epsilon }(z_{i}) \neq \emptyset $, and denote the minor arc between their two intersection points as $\zeta $ and its length as $\ell $. Now we choose $k$ so that $2^{k}\ell > 2\pi $. Since the sequence $\lim_{j \to +\infty} {\lambda _{j}} = 0$, thus when $z$ lies outside the circle outside the circle centred at the origin with radius $\dfrac{|p|}{2}$, we can choose a sufficiently large j such that $\left| F_{\lambda _{j}}^{m}(z) - z^{2^{m}} \right| $ is extremely small. Thus, the argument of the curve $F_{c}(\zeta )$ increases by $2\pi $ approximately, and therefore the curve $F_{c}^{k}(\zeta )$ wraps around the origin at least once. Hence this curve must meet the Cantor necklace in the dynamical plane. However, the Cantor necklaces are always located in a subset of the Julia set, which implies that the curve $F_{c}^{k}(\zeta )$ must intersect with the Julia set $J(F_{\lambda _{j}})$. Since the Julia set is backward invariant (that is, $F_{\lambda }^{-1} \left( J(F_{\lambda _{j}}) \right) \subset J(F_{\lambda _{j}})$), then we know that $J(F_{\lambda _{j}}) \cap B_{\epsilon }(p) \neq \emptyset $, which is a contradiction, and therefore the theorem is now proved.
\end{proof}

\item At the end of this section, we want to mention the dynamics when the nonholomorphic singular perturbation is introduced in the case of $m=2$; i.e. $G_{\beta}: \mathbb{C}\rightarrow\mathbb{C}$, where $G_{\beta}(z) = z^{2} + \dfrac{\beta}{\bar z^{2}}$. Similar to the case under holomorphic singular perturbation, this family is the most complicated one in nonholomorphic singular perturbation as well. Both this family and a more general form (the radial symmetry case), $G_{\beta, m}(z) = z^{m} + \dfrac{\beta}{\bar {z}^{m}}$, had been well studied by B. Peckham and B. Bozyk \cite{Peckham II}. 

\clearpage
\subsection{Singular Perturbations of Quadratic Family when $m \geq 3$ and $m \in \mathbb{N}$}
\item Similar to the previous cases, the family when $m \geq 3$ have $(2 + m)$-fold symmetry, and let $\omega = sin \left( \dfrac{2\pi }{2+m} \right) + i cos \left( \dfrac{2\pi }{2+m} \right) $ be the $(2 + m)$-th root of the unity, the following equality holds:
\begin{equation*}
F_{\lambda , m}(\omega z) = {\omega }^{2} F_{\lambda , m}(z).
\end{equation*}
Furthermore, one can easily show that the saddle-node bifurcation value $\lambda $ for any $m \in \mathbb{N}$ (including the cases $m=1,2$) takes the following form:
\begin{equation*}
\lambda_{0} = \dfrac{\left( {1 + m}\right) ^{1 + m}}{\left( {2 + m}\right) ^{2 + m}}.
\end{equation*}
Meanwhile, the first derivative (for any $m \in \mathbb{N}$), the first derivative is 
\begin{equation*}
F^{'}_{\lambda , m}(z) = 2z - (m\lambda) z^{-(m+1)} = (2 + m)z - \dfrac{m}{z} F_{\lambda , m}(z).
\end{equation*}
from which one can derive that the critical points are $ z_{c} = \left( \dfrac{m\lambda }{2}\right)^{1/(m+2)}$. 
Therefore, when z is a fixed point, it should satisfy 
\begin{equation*}
F^{'}_{\lambda , m}(z) = (2 + m)z - m. 
\end{equation*}

\begin{figure}[!htb]
\centering
\centerline{
\includegraphics[scale=0.35]{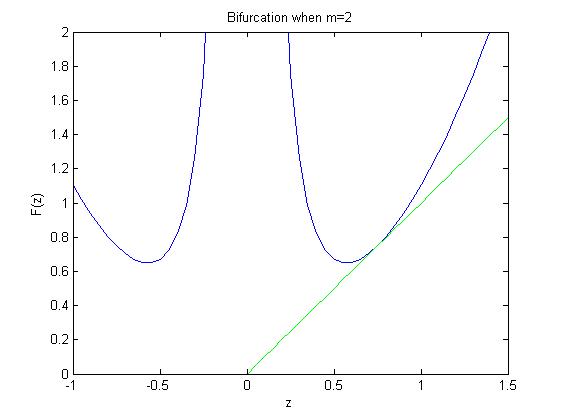}
\includegraphics[scale=0.35]{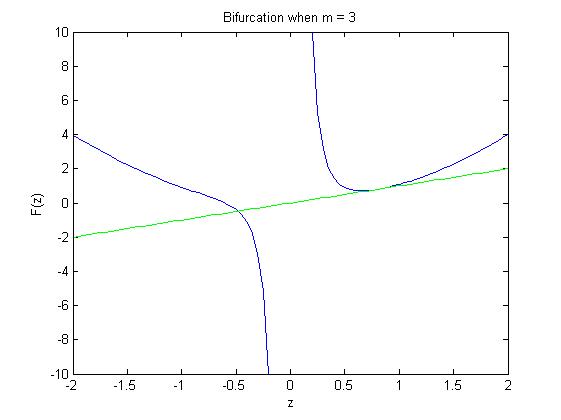}
\includegraphics[scale=0.35]{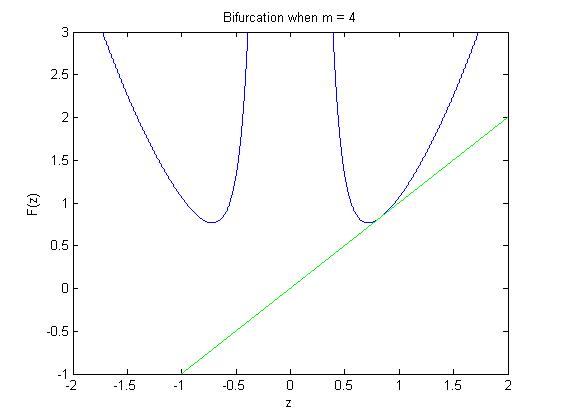}}
\centerline{
\includegraphics[scale=0.35]{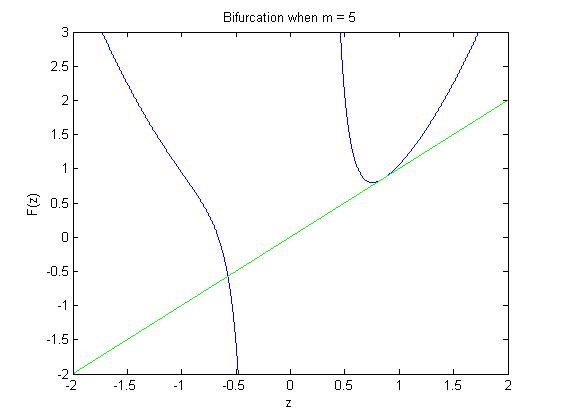}
\includegraphics[scale=0.35]{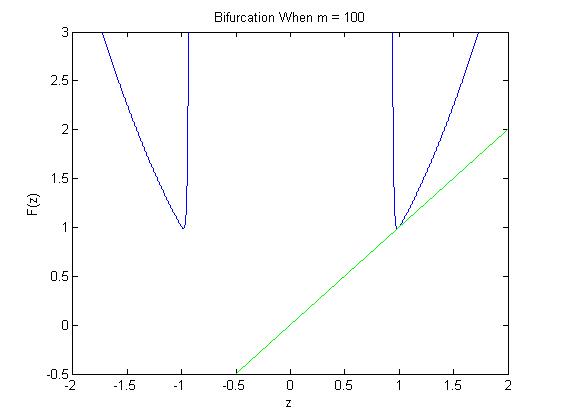}
\includegraphics[scale=0.35]{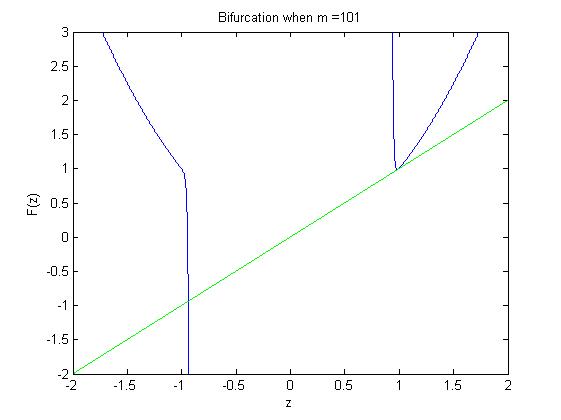}}
\caption{\small The Saddle-Node Bifurcations of $F_{\lambda}: \mathbb{C}\rightarrow\mathbb{C}$, where $F_{\lambda , m}(z) = z^{2} + \dfrac{\lambda}{z^{m}}$, in which $\lambda \in \mathbb{R}$ when $m = 2,3,4,5,100,101$.}
\end{figure}

\item We start our discussion with a relatively simple case when $\lambda \in \mathbb{R}$ and $\lambda \geq 0$. Let us denote the saddle node in this case as $\lambda _{0}$. According to Sturm's Theorem, one can show that there are at most two non-negative real fixed points, let us denote them as $z_{1}$ and $z_{2}$, where $z_{1} < z_{2}$. Then it can be readily shown that, for $\lambda \in (0, \lambda _{0})$, $|F^{'}_{\lambda , m}(z_{1})| \leq 1$ (that is, $z_{1}$ is an attracting point) if and only if $m=1$. In the following parts, we will discuss the cases when $m \geq 3$, and the corresponding figures are shown in Fig.13.
 
\item In the previous section, we referred that the dynamics when $m = 2$ on complex plane is the most complicated case. This is still true when reduced to the real line $\mathbb{R}$. Thus, in the following we start with the simpler case $m \geq 3$ and still restrict our discussion into $\lambda \in [0, \lambda _{0}]$, then consider the dynamics when $m = 2$ later.  

\item Let us firstly consider its dynamics when $\lambda $ is close to the two boundary points 0 and $\lambda _{0}$ when $m =3$. Firstly, we denote the positive preimage of $z_{2}$ as $s$ (that is, $F_{\lambda , m}(s) = s^{2} + \dfrac{\lambda}{s^{m}} = z_{2}$). When $\lambda \in [0, \lambda _{c}]$ is close to 0, we have $z_{2} < 1$ for $\lambda > 0$, and then $s^{2} + \dfrac{\lambda }{s^{m}} < 1$, which implies that $\dfrac{\lambda }{s^{m}} < 1$ and therefore $s^{m} < \lambda $. Meanwhile, since 
\begin{equation*}
F_{\lambda }(z_{c}) = \lambda ^ {2/(2+m)} \left( (\dfrac{m^{2}}{4})^{1/(2+m)} + (\dfrac{2^{m}}{m^{m}})^{1/(2+m)} \right),
\end{equation*} 
and $2/(2+m) > 1/m$ (which holds for all 
$m \geq 3$), then $F_{\lambda , m}(z_{c})$ (which is order $\lambda ^{2/(2+m)}$) is less than $s$ (which is at least order $\lambda ^{1/m}$) when for $\lambda $ approaches to 0. Since $F_{\lambda , m}(z_{c})$ decreases in the interval $(0, z_{c}]$, then 
we can conclude that when $\lambda $ is close to 0, the following equality holds:
\begin{equation*}
F_{\lambda ,m}^{2}(z_{c}) > F_{\lambda ,m}(s) = z_{2}. 
\end{equation*}
On the other hand, when $\lambda \in [0, \lambda _{c}]$ is close to $\lambda _{0}$, $z_{c} < z_{1}$ and $F'_{\lambda }(z_{c}) >0 $ for all $z \in [z_{c}, z_{1}]$. Then one can readily show that $z_{c} < F_{\lambda ,m}(z_{c}) < F_{\lambda ,m}^{2}(z_{c}) < z_{1} < z_{2}$, which means that $F_{\lambda ,m}^{2}(z_{c}) < z_{2}$ when $\lambda $ is close to $\lambda _{0}$.
Based on the above discussion, since $F_{\lambda ,m }^{2}(z_{c}) > z_{2}$ when $\lambda $ is close to 0 while $F_{\lambda ,m}^{2}(z_{c}) < z_{2}$ when $\lambda $ when $\lambda $ is close to $\lambda _{0}$, then by the Intermediate Value Theorem, there exists a $\lambda _{iv} \in [0, \lambda _{0}]$ such that $F_{\lambda _{iv}, m}(z_{c}) = z_{2}$. Then, by a theorem proved in \cite{Douady} (see the following remark), we can conclude that a period-doubling bifurcation occurs on $\mathbb{R}^{+}$ when $\lambda $ decreases from $\lambda _{0}$ to $0$.

\begin{remark}
This theorem claims: if there is a parameter value $\lambda $ such that $F_{\lambda }^{2}(z_{c})$ equals the repelling fixed point, then a period doubling bifurcation will occur. See \cite{Douady} for proof. 
\end{remark}

\item Finally let us consider the most complicated case when $m = 2$, in which the critical points are $ z_{c} = \lambda ^{1/4}$, where $\lambda \in \mathbb{R}$ and saddle node bifurcation occurs when $\lambda _{0} = 27/ 256$. Then the value of $z_{c}$ after $k-th$ iterations is  
\begin{equation*}
F_{\lambda , m}^{k}(z_{c}) = (z_{c}^{2} + \dfrac{\lambda }{z_{c}^{2}})^{k} = (2 \lambda ^{1/2})^{k}.
\end{equation*}
Since $F_{\lambda , m}^{k}(z_{c})$ monotonically decreases when $k$ increases, the maximum value on the orbit of the critical point is $F_{\lambda , m}^{k}(z_{1}) = 2 \lambda ^{1/2}$. Therefore, when $m=2$. the critical orbit never escapes the interval $[s, z_{2}]$. And similar to the cases when $m \geq 3$, there is a period doubling bifurcation appears on when $\mathbb{R}$ when $\lambda $ decreases from $\lambda _{0}$ to 0.  

\item At the end of this section, we summarize the three ceases discussed above and conclude the following theorem:
\begin{theorem}
There exist at most two non-negative fixed points $z_{1}$ and $z_{2}$ ($z_{1} < z_{2}$) for the family $F_{\lambda}: \mathbb{C}\rightarrow\mathbb{C}$, where $F_{\lambda , m}(z) = z^{2} + \dfrac{\lambda}{z^{m}}$, in which $\lambda \in \mathbb{R}$, $m \in \mathbb{N}$, and \\
(1)For all $\lambda \in (0, \lambda _{0})$, $z_{1}$ is attracting if and only if m = 1;\\
(2)When $m = 2$, the orbit of the critical point $z_{c}$ never escapes the interval $[s, z_{2}]$, and a period doubling bifurcation occurs on $\mathbb{R}^{+}$ when $\lambda $ decreases from $\lambda _{0}$ to $0$;\\
(3)When $m \geq 3$, for sufficient small $\lambda $, the orbit of the critical point $z_{c}$ will escapes the interval $[s, z_{2}]$, and a period doubling bifurcation occurs on $\mathbb{R}^{+}$ when $\lambda $ decreases from $\lambda _{0}$ to $0$.
\qed
\end{theorem}

\clearpage

\subsection{Simple Comparison of Holomorphic and Nonholomorphic Singular Perturbations}

\item Before ending this section, we want to show the escape figures of parameter planes (Fig.14) and some escape figures of dynamical planes (Fig.15) of $F_{\lambda}(z) = z^{2} + \dfrac{\lambda}{z}$ and $G_{\beta}(z) = z^{2} + \dfrac{\beta}{\bar z}$ , where $\lambda,\beta \in \mathbb{C}$. One should note that $F_{\lambda}$ is a map from $\mathbb{C}$ to $\mathbb{C}$ while $G_{\beta}$ is a map from $\mathbb{R}^{2}$ to $\mathbb{R}^{2}$, in which the former is a very special case of the latter \cite{Peckham II}. And it is worth to mention that the real case we discussed at the beginning of \textbf {section 3}, i.e. $F_{c}: \mathbb{R}\rightarrow\mathbb{R}$, where $F_{c}(x) = x^{2} + \dfrac{c}{x}$, in which $c \in \mathbb{R}$, is along the spine (real line) of the escape figures of dynamic planes; for parameter planes, however, its dynamics indeed matches the whole spine in the holomorphic case, but only matches the positive parts of the spine in the nonholomorphic case.     

\item In Fig.14 we see that the parameter planes of these two families are entirely different. The left graph in Fig.14 is a \textbf {Pseudo-Mandelbrot set}, in which there exist infinitely many parts that the Mandelbrot set is topologically equivalent to. See \cite{Devaney VIIII} for more detailed discussion about this dynamical plane. The right graph in Fig.14, however, is far away from being well understood. 

\begin{remark}
A Pseudo-Mandelbrot set for the map $F_{\lambda}(z) = z^{2} + \dfrac{\lambda}{z}$ is a collection of $\lambda$-values for which the critical orbits under function $F_{\lambda}(z)$ stay bounded. 
\end{remark}

\item For each $\lambda$ in the holomorphic singular perturbation, there exist three critical points, which are the roots of equation $F_{\lambda}^{'}(z) = 0$. As we proved in \textbf {section 3.3}, these three critical points possess the same dynamics (either stay bounded or go off to infinity), therefore anyone of these critical points will produce the same escape figure. Therefore, the left graph in Fig.14 is called ``the" parameter plane.    

\item In the case of nonholomorphic singular perturbation ($G_{\beta}$), however, the set of critical points is a circle of radius $(|\beta|/2)^{1/3}$ \cite{Peckham II}. On the other hand, the roots of $F_{\lambda}^{'}(z) = 0$ are the three values of $(\lambda /2)^{1/3}$ (when $\lambda \in \mathbb{C}$), and all of these three critical points have the same magnitude $(|\lambda|/2)^{1/3}$. These imply that the three roots of $F_{\lambda}^{'}(z) = 0$ in holomorphic case lie on the critical circle in the nonholomorphic case (except for $\lambda = \beta =0$). There is no surprise that we obtain such a result, because as we pointed out above that the complex plane is a subset of the two-dimensional real plane. Since the critical points on the critical circle do not necessarily all possess the same dynamics, the escape figures of dynamical planes are not unique and definitive. This is the reason why the right graph in Fig.14 is called ``a" parameter plane escape figure in the nonholomorphic case. It is also necessary to point out that the right graph in Fig.14 was plotted by using the positive real point on the critical circle, which is also a critical point for $F_{\lambda}$ when $\lambda$ is a positive real number. This explains why the points on the positive horizontal (real) axes in the two escape figures in Fig.14 share the same dynamics. However, this point is not a critical point for $F_{\lambda}$ when $\lambda$ is a negative real number, so these two escape figures do not necessarily agree on the negative real axes. 

\item So far, we explained some differences in the escape figures of parameter planes under holomorphic and nonholomorhic singular perturbations. However, understanding more details about their dynamics, such as how parameter planes in the nonholomorphic case depend on which critical point is selected, and whether some of the points on the critical circle share the same escape property as the three critical points in the holomorphic case, requires more studies.

\begin{figure}
\begin{tabular}{llll}
& \includegraphics[height=7cm,width=7cm]{parameter1.jpg} & \includegraphics[height=7cm,width=7cm]{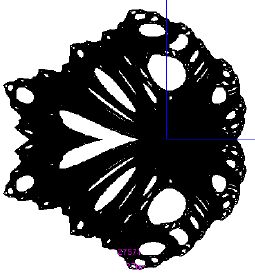}\\
&$F_{\lambda}(z) = z^{2} + \dfrac{\lambda}{z}$ (Holomorphic Singular Perturbation) & $G_{\beta }(z) = z^{2} + \dfrac{\beta }{\bar z}$(Nonholomorphic Singular Perturbation) 
\end{tabular}
\caption{Escape Figures of the Parameter Planes of $F_{\lambda}(z) = z^{2} + \dfrac{\lambda}{z}$ and a Parameter Plane of $G_{\beta}(z) = z^{2} + \dfrac{\beta}{\bar z}$, where $\lambda, \beta \in \mathbb{C}$. (Note the difference between ''the" and ''a" in this caption.)}
\end{figure}

\item Fig.15 shows the escape figures of dynamic planes under three typical values $\lambda$ (or $\beta$): 4/27, -0.327, and -0.507. Since the real axis is invariant under $F_{\lambda}$ when $\lambda \in \mathbb{R}$ and $G_{\beta}$ when $\beta \in \mathbb{R}$, and $F_{\lambda} = G_{\beta}$ when restricted to the $x$-axis, the graphs in the left column and the graphs in the right column should each agree along the $x$-axis. This is reasonably clear when the parameters are 4/27 and -0.507, but not for the nonholomorphic case when the parameter equals -0.327. In the next parts, we are trying to quantitatively interpret the difference in the escape figures under holomorphic versus nonholomorphic singular perturbations when $\lambda = \beta = -0.327$.

\item In the period-three case, the corresponding $\lambda$ (and $\beta$)-value and critical point $z_{0}$ are approximately -0.327 and -0.549241, respectively; and therefore $z_{1} = G_{\beta} (z_{0}) \thickapprox 0.897033$ and $z_{2} = G_{\beta} (z_{1}) \thickapprox 0.440311$. These three points, $z_{0}$, $z_{1}$, and $z_{2}$, which form a period-three cycle, respectively lie in the three black blobs along the spine (the middle-left graph in Fig.15). However, in the middle-right graph in Fig.15 (the case of nonholomorphic singular perturbations), no black blobs appears along the spine, although the three periodic points indeed exits (i.e. at least three black dots should appear). 

\item Now we write $G_{\beta}$ in $x-y$ coordinates
\begin{equation*}
G_{\beta}(x,y) = (x^2 - y^2 + \frac{\beta x}{x^2 + y^2}) + (2xy + \frac{\beta y}{x^2 + y^2})i,
\end{equation*}
in which we denote the real part as $G_{\beta R} = Re\{G_{\beta}(x,y) \} = x^2 - y^2 + \frac{\beta x}{x^2 + y^2}$ and the imaginary part as $G_{\beta I} = Im\{G_{\beta}(x,y) \} = 2xy + \frac{\beta y}{x^2 + y^2}$. Then, one can show that the corresponding Jacobian matrix in $z-\bar{z}$ coordinates is 
\begin{equation*}
\dfrac{\partial (G_{\beta R},G_{\beta I})}{\partial x \partial y} = \left( \begin{array}{cc} 2x - \frac{\beta (x^2 - y^2)}{{(x^2 + y^2)}^2} & -2y - \frac{2\beta xy}{{(x^2 + y^2)}^2} \\ 2y - \frac{2\beta xy}{{(x^2 + y^2)}^2} & 2x + \frac{\beta (x^2 - y^2)}{{(x^2 + y^2)}^2} \end{array} \right).
\end{equation*}
On the $x$-axis, $y=0$ and this Jacobian matrix can be rewritten as 
\begin{equation*}
\dfrac{\partial (G_{\beta R},G_{\beta I})}{\partial x \partial y}|_{y=0} = \left( \begin{array}{cc} 2x - \frac{\beta}{x^2} & 0 \\ 0 & 2x + \frac{\beta}{x^2} \end{array} \right).
\end{equation*}

\item Then, by the chain rule, we can determine the Jacobian matrix of $({G_{\beta R}}^{3},{{G_{\beta I}}^{3}})$ at the critical period-three point:
\begin{equation*}
\begin{aligned}
\dfrac{\partial {(G_{\beta R}}^{3}, {G_{\beta I}}^{3})}{\partial x \partial y} = \dfrac{\partial (G_{\beta R}, G_{\beta I})}{\partial x \partial y} |_{z=z_{2}}\cdot \dfrac{\partial (G_{\beta R}, G_{\beta I})}{\partial x \partial y} |_{z=z_{1}} \cdot \dfrac{\partial (G_{\beta R}, G_{\beta I})}{\partial x \partial y} |_{z=z_{0}} \thickapprox \left( \begin{array}{cc} -0.081916 & 0 \\ 0 & 2.44116 \end{array} \right),
\end{aligned}
\end{equation*}
the eigenvalues of which are approximately 0 and 2.4. 

\item In the $x-y$ coordinates (the coordinates we used in Fig.15, the eigenvalues of the Jacobian matrix can roughly explain the dynamics near the period-three orbit. It is obvious that in period-three case the eigenvector corresponding to the eigenvalue 0 is parallel to the x-axis, which indicates the appearance of the critical point. The other eigenvalue, 2.4 ($|2.4|>1$), with the eigenvector parallel to the y-axis, tells us that the dynamics along this direction is repelling. Therefore, although the period-three cycle is attracting when restricted to the $x$-axis, it becomes a saddle point in the $x-y$ plane. And all points (including those lying in the neighborhood of three periodic points) with a nonzero imaginary part will initially head away from the $x$-axis after iteration. This partially explains why there is no black blobs appear along the spine in the middle-right graph in Fig.15. 

\begin{remark} Actually, if we extend $G_{\beta}$ to $(G_{\beta}, \overline{G_{\beta}})$, where $\overline{G_{\beta}} = {\bar z}^{2} + \dfrac{\bar \beta}{z}$, then we can compute the Jacobian matrix in $z-\bar{z}$ coordinate more easily \cite{Peckham II}. The Jacobian matrix of $G_{\beta}, \overline{G_{\beta}})$ is definitely different from the Jacobian matrix of $(F_{\lambda}, \overline{F_{\lambda}})$, which is obtained by extending $F_{\lambda}$. However, one eigenvalue of zero corresponding to the eigenvector along the $x$-axis must be obtained in both of these two Jacobian matrices, because these two extensions both include the points along the $x$-axis. However, what will happen on the eigenvalues in other directions that are transverse to the $x$-axis depends on the way we extend the map. 
\end{remark}

\item It worth to mention that one can show that the period-four cycle (when $\lambda = \beta = -0.507$) consists of four points: -0.632282 (critical point), 1.201797, 1.022799, and 0.550833; and they respectively lie in the four black blobs along the spine (the lower-left graph in Fig.15). Then, following the same approach, one can readily show that the eigenvalues of the Jacobian matrix for the period-four case (when $\beta = -0.507$) are approximately 4.60841 and 0. Although the eigenvalue corresponding to the eigenvector that is parallel to the $y$-axis is 0, some black blobs that contain these points on the critic orbit still appear in the lower-right graph in Fig.15. This indicates that we still do not completely understand the dynamics in the neighborhoods of the points on the period-four cycle.

\item Finally, it is necessary to point out we only explored some simple phenomena appear in the escape figures under nonholomorphic singular perturbations, even for the period-three and -four cases; more deeper studies are required for a better understanding at these escape figures.

\begin{figure}
\begin{tabular}{llll}
& \includegraphics[height=6cm,width=6cm]{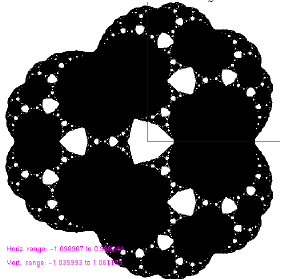} & \includegraphics[height=6cm,width=6cm]{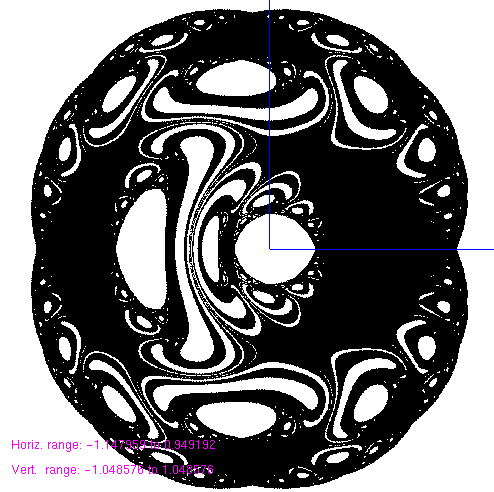}\\
& \small {Dynamic Plane of $F_{\lambda}(z) = z^{2} + \dfrac{4/27}{z}$ (Saddle Node)} & \small {Dynamic Plane of $G_{\beta}(z) = z^{2} + \dfrac{4/27}{\bar z}$(Saddle Node)}\\
& \includegraphics[height=6cm,width=6cm]{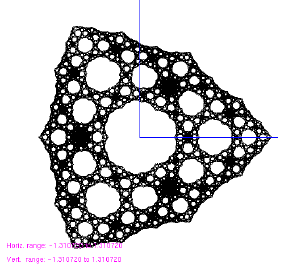} & \includegraphics[height=6cm,width=6cm]{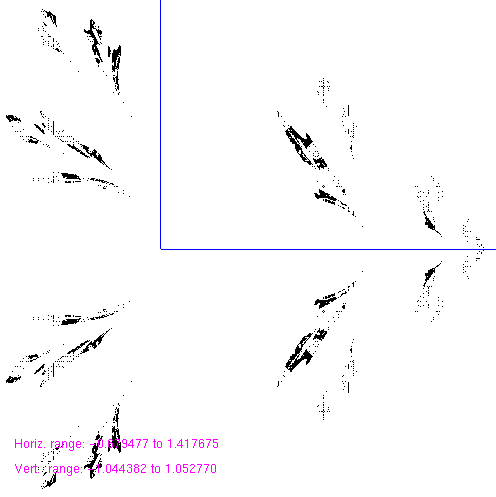}\\
& \small {Dynamic Plane of $F_{\lambda}(z) = z^{2} + \dfrac{-0.327}{z}$ (Period Three)} & \small{Dynamic Plane of $G_{\beta}(z) = z^{2} + \dfrac{-0.327}{\bar z}$(Period Three)} \\
& \includegraphics[height=6cm,width=6cm]{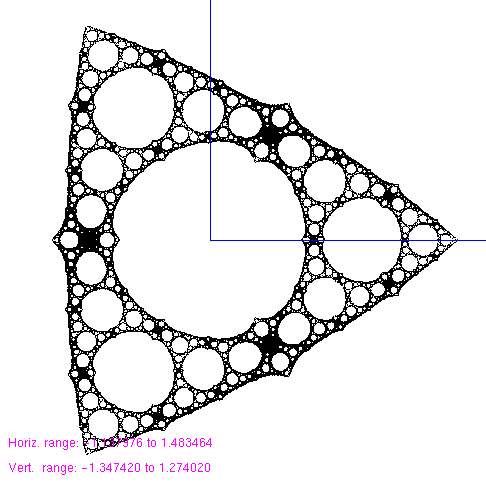} & \includegraphics[height=6cm,width=6cm]{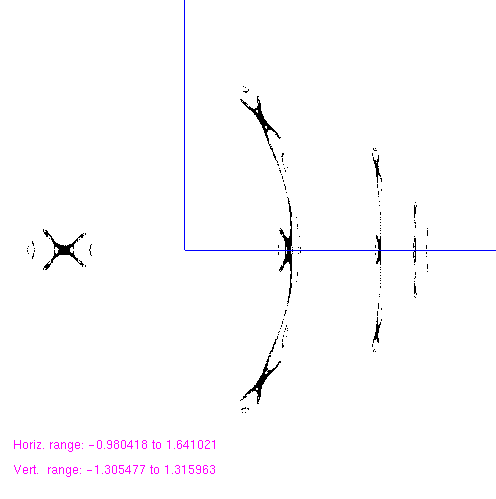}\\
& \small{Dynamic Plane of $F_{\lambda}(z) = z^{2} + \dfrac{-0.507}{z}$ (Period Four)} & \small{Dynamic Plane of $G_{\beta}(z) = z^{2} + \dfrac{-0.507}{\bar z}$(Period Four)} \\
\end{tabular}
\caption{The Escape Figures of Several Typical Dynamic Planes of $F_{\lambda}(z) = z^{2} + \dfrac{\lambda}{z}$ and $G_{\beta}(z) = z^{2} + \dfrac{\beta}{\bar z}$, where $\lambda, \beta \in \mathbb{C}$}
\end{figure}

\clearpage

\section{Summary}
\item In this paper, we summarized some older and some more recent studies on both real and complex quadratic families, and researched some details of the quadratic family under nonholomorphic singular perturbation with the form of $G_{\beta}:\mathbb{R}^{2}\rightarrow\mathbb{R}^{2}$, where $G_{\beta }(z) = z^{2} + \dfrac{\beta}{\bar z}$, in which $\beta \in \mathbb{C}$ (especially its dynamics along the real line). As we mentioned before, the complex families we discussed in this paper are the simplified cases of two more general families: $F_{\lambda,c,n,m}:\mathbb{C}\rightarrow\mathbb{C}$, where $F_{\lambda,c,n,m}(z) = z^{n} + c + \dfrac{\lambda}{z^{m}}$, and $G_{\beta,c,n,m}:\mathbb{R}^{2}\rightarrow\mathbb{R}^{2}$, where $G_{\beta,c,n,m}(z) = z^{n} + c + \dfrac{\beta}{\bar{z} ^{m}}$, in both of which $n,m \in \mathbb{N}$ and $\lambda, \beta, c \in \mathbb{C}$). However, there exists a much more general family: $Y_{\lambda, \beta, \alpha, c, n,  m_{1}, m_{2}, d}:\mathbb{R}^{2}\rightarrow\mathbb{R}^{2}$, where $Y_{\lambda, \beta, \alpha, c, n,  m_{1}, m_{2}, d}(z) = z^{n} + c + \dfrac{\lambda}{z^{m_{1}}} + \dfrac{\beta}{\bar z^{m_{2}}} + \alpha \bar{z}^{d} $, in which $n,m_{1},m_{2},d \in \mathbb{N}$ and $\lambda, \beta, \alpha, c \in \mathbb{C}$ \cite{Peckham II}, from which we took the families mentioned above. In this paper, readers probably have felt both the complication and elegancy of the dynamics of these simplified cases, and could imagine the difficulties we will probably encounter in the future research on these seemingly simple maps. More studies are definitely required for both better understanding the families we have referred in this paper and interpreting those more complicated and general families.

\end{document}